\documentclass[letterpaper, 11pt, reqno]{amsart}

\usepackage{amsmath,amssymb,amscd,amsthm,amsxtra, esint}

\usepackage[implicit=true]{hyperref}
\allowdisplaybreaks[2]

\usepackage[left=32mm, right=32mm, 
bottom=27mm]{geometry}

\usepackage{color}

\definecolor{gr}{rgb}   {0.,   0.69,   0.23 }
\definecolor{bl}{rgb}   {0.,   0.5,   1. }
\definecolor{mg}{rgb}   {0.85,  0.,    0.85}
\definecolor{yl}{rgb}   {0.8,  0.7,   0.}
\definecolor{or}{rgb}  {0.7,0.2,0.2}

\newtheorem{theorem}{Theorem} [section]

\newtheorem{lemma}[theorem]{Lemma}
\newtheorem{proposition}[theorem]{Proposition}
\newtheorem{remark}[theorem]{Remark}

\newcommand{\noi}{\noindent}
\newcommand{\Z}{\mathbb{Z}}
\newcommand{\R}{\mathbb{R}}
\newcommand{\C}{\mathbb{C}}
\newcommand{\T}{\mathbb{T}}

\let\Re=\undefined\DeclareMathOperator*{\Re}{Re}
\let\Im=\undefined\DeclareMathOperator*{\Im}{Im}

\let\P= \undefined
\newcommand{\P}{\mathbf{P}}

\newcommand{\E}{\mathbb{E}}

\renewcommand{\L}{\mathcal{L}}

\newcommand{\D}{\mathcal{D}}

\newcommand{\F}{\mathcal{F}}

\newcommand{\al}{\alpha}
\newcommand{\be}{\beta}
\newcommand{\dl}{\delta}

\newcommand{\nb}{\nabla}

\newcommand{\Dl}{\Delta}
\newcommand{\eps}{\varepsilon}
\newcommand{\kk}{\kappa}
\newcommand{\g}{\gamma}
\newcommand{\G}{\Gamma}
\newcommand{\ld}{\lambda}

\newcommand{\s}{\sigma}

\newcommand{\ft}{\widehat}

\newcommand{\wt}{\widetilde}
\newcommand{\cj}{\overline}

\newcommand{\dt}{\partial_t}
\newcommand{\dd}{\partial}

\renewcommand{\l}{\ell}
\renewcommand{\o}{\omega}
\renewcommand{\O}{\Omega}

\newcommand{\les}{\lesssim}
\newcommand{\ges}{\gtrsim}

\newcommand{\jb}[1]
{\langle #1 \rangle}

\newcommand{\ind}{\mathbf 1}

\newcommand{\N}{\mathbb{N}}

\newcommand{\NN}{\mathcal{N}}

\newcommand{\EE}{\mathcal{E}}

\renewcommand{\H}{\mathcal{H}}

\newcommand{\QQ}{\mathcal{Q}}

\newtheorem*{ackno}{Acknowledgments}

\numberwithin{equation}{section}
\numberwithin{theorem}{section}

\newcommand{\lin}{\textup{lin}}

\usepackage{tikz} 
\usepackage{marginnote}
\usepackage{scalerel} 

\usetikzlibrary{shapes.misc}
\usetikzlibrary{shapes.symbols}
\usetikzlibrary{decorations}
\usetikzlibrary{decorations.markings}

\tikzset{
	dot/.style={circle,fill=black,draw=black,inner sep=0pt,minimum size=0.5mm},
	>=stealth,
	}

\tikzset{
	dot2/.style={circle,fill=black,draw=black,inner sep=0pt,minimum size=0.2mm},
	>=stealth,
	}

\tikzset{
	ddot/.style={circle,fill=white,draw=black,inner sep=0pt,minimum size=0.8mm},
	>=stealth,
	}

\tikzset{decision/.style={ 
        draw,
        diamond,
        aspect=1.5
    }}

\tikzset{dia2/.style
={diamond,fill=white,draw=black,inner sep=0pt,minimum size=1mm},
	>=stealth,
	}

\tikzset{dia/.style
={star,fill=black,draw=black,inner sep=0pt,minimum size=1mm},
	>=stealth,
	}

\tikzset{dia/.style
={diamond,fill=black,draw=black,inner sep=0pt,minimum size=1.3mm},
	>=stealth,
	}

\makeatletter

\def\<#1>{\xusebox{#1}}
\makeatother

\newcommand{\too}{\longrightarrow}

\begin{document}
\baselineskip = 14pt

\title[A remark on triviality for 2d SNLW]
{A remark on triviality for the two-dimensional stochastic nonlinear wave equation}

\author[T. Oh]{Tadahiro Oh}
\address{
Tadahiro Oh\\
School of Mathematics\\
The University of Edinburgh\\
and The Maxwell Institute for the Mathematical Sciences\\
James Clerk Maxwell Building\\
The King's Buildings\\
 Peter Guthrie Tait Road\\
Edinburgh\\ 
EH9 3FD\\United Kingdom} 

\email{hiro.oh@ed.ac.uk}

\author[M. Okamoto]{Mamoru Okamoto}
\address{Mamoru Okamoto\\
Division of Mathematics and Physics, Faculty of Engineering, Shinshu University, 4-17-1 Wakasato, Nagano City 380-8553, Japan}
\email{m\_okamoto@shinshu-u.ac.jp}

\author[T. Robert]{Tristan Robert}
\address{
Tristan Robert\\
School of Mathematics\\
The University of Edinburgh\\
and The Maxwell Institute for the Mathematical Sciences\\
James Clerk Maxwell Building\\
The King's Buildings\\
 Peter Guthrie Tait Road\\
Edinburgh\\ 
EH9 3FD\\United Kingdom} 

 \curraddr{
Fakult\"at f\"ur Mathematik\\
Universit\"at Bielefeld\\
Postfach 10 01 31\\
33501 Bielefeld\\
Germany}

\email{trobert@math.uni-bielefeld.de}

\begin{abstract}
We consider  the two-dimensional stochastic damped nonlinear wave equation (SdNLW) 
with the cubic nonlinearity, 
forced by a space-time white noise. 
In particular, we investigate the limiting behavior of solutions
to SdNLW with regularized  noises
and establish triviality results
in the spirit of the work by Hairer, Ryser, and Weber (2012).
More precisely, without renormalization of the nonlinearity, 
we establish the following two limiting behaviors;
(i) in the strong noise regime, 
we show that  solutions to SdNLW
with regularized noises  tend  to $0$ as the regularization is removed
and (ii) in the weak noise regime, 
we  show that  solutions to SdNLW with regularized noises
converge to a solution to a deterministic damped nonlinear wave equation 
with  an additional mass term.
\end{abstract}

%
\maketitle
%


\section{Introduction}\label{SEC:1}
\subsection{Stochastic damped nonlinear wave equation, 
renormalization, and triviality} 

We consider the Cauchy problem for the following stochastic damped nonlinear wave equation (SdNLW) 
with the cubic nonlinearity, posed 
on the two-dimensional torus $\T^2 = (\R/2\pi \Z)^2$:
\begin{equation}
\label{SdNLW}
\begin{cases}
 \dt^2 u - \Dl  u + \dt u + u^3 = \al \xi \\
 (u,\dt u)|_{t = 0} = (u_0,u_1)
\end{cases}
\qquad (t, x) \in \R_+\times \T^2,
\end{equation}

\noi
where $\al \in \R$  and $\xi (t,x)$ denotes a space-time white noise on $\R_+ \times \T^2$.
The damped wave equation (without a stochastic forcing)
appears as a model describing  wave propagation with friction.
It also appears 
as  a modified heat conduction equation with the finite propagation speed property
\cite{Cat} 
and as  stochastic models such as correlated random walk~\cite{Kac}.
See~\cite{IIW}
for further references.
In the deterministic case, the equation \eqref{SdNLW} has been studied
extensively; see
\cite{HKN, IIW, IIOW} and the references therein.

The stochastic nonlinear wave equations (SNLW)
 have been studied extensively
in various settings; 
see \cite[Chapter 13]{DPZ14} for the references therein.
In recent years, we have witnessed
a rapid progress on the theoretical understanding 
of SNLW with singular stochastic forcing.
 In~\cite{GKO},  Gubinelli, Koch, 
and  the first author
considered SNLW with an additive space-time white noise on $\T^2$:
\begin{equation}
\label{SNLW}
 \dt^2 u - \Dl  u  + u^k =  \xi, 
\end{equation}

\noi
where $k\geq 2$ is an integer.
The main difficulty of this problem 
already appears in the stochastic convolution $\Psi$, 
solving the linear equation:
\begin{equation}
\label{SNLW2}
 \dt^2 \Psi - \Dl  \Psi   =  \xi.
\end{equation}

\noi
It is well known that for the spatial dimension $d \geq 2$, 
the stochastic convolution  $\Psi$ 
is not a classical function but is merely a Schwartz distribution.
 In particular, there is an issue in making sense of powers $\Psi^k$ and a fortiori of the full nonlinearity 
 $u^k$ in \eqref{SNLW}.
This requires us to modify the equation in order to take into account a proper {\it renormalization}.

In \cite{GKO}, by introducing appropriate time-dependent
renormalization, 
the authors proved local well-posedness of (a renormalized version of) \eqref{SNLW} on $\T^2$.
In \cite{GKOT} with Tolomeo, 
they constructed global-in-time dynamics for \eqref{SNLW}
in the cubic case ($k = 3$).
The local well-posedness argument in \cite{GKO}
essentially  applies to SdNLW \eqref{SdNLW}
with a general  power-type nonlinearity~$u^{k}$.
When $\al = \sqrt2$, 
the equation \eqref{SdNLW} 
corresponds to the so-called canonical stochastic quantization for the $\Phi^4_2$-measure in Euclidean quantum field theory (see \cite{RSS}), 
which 
formally preserves the Gibbs measure\footnote{Namely, the $\Phi^4_2$-measure on 
the $u$-component coupled with 
the white noise measure on the $\dt u$-component.}
for the deterministic nonlinear wave equation
studied in \cite{OTh2}.\footnote{Strictly speaking, 
the results mentioned here only apply 
to the nonlinear Klein-Gordon case, i.e.~$-\Dl$ in~\eqref{SdNLW} replaced by $1 - \Dl$.}
By combining the local well-posedness
argument 
with Bourgain's invariant measure argument
\cite{BO94, BO96}, 
it was shown in~\cite{GKOT}
that SdNLW \eqref{SdNLW}, 
with a general defocusing power-type nonlinearity $u^{2k+1}$, 
is almost surely globally well-posed
with the Gibbs measure initial data
and that the Gibbs measure is invariant under the dynamics.
We also mention a recent extension  \cite{ORTz} of these results
to the case of two-dimensional compact Riemannian manifolds
without boundary
and a recent work \cite{GKO2} in establishing local well-posedness
of the quadratic SNLW on the three-dimensional torus $\T^3$.

In the works mentioned above, 
renormalization played an essential role, 
allowing us to give a precise meaning to the equations.  
Our main goal in this paper is to study
the behavior of solutions to \eqref{SdNLW}, 
in a suitable limiting sense, 
{\it without} renormalization.
Namely, we consider 
the equation \eqref{SdNLW}
with a regularized noise, 
 via frequency truncation, 
and study possible limiting behavior of solutions
as we remove the regularization.
In particular, 
we establish a triviality
result
in a certain regime;
as we remove the regularization, 
solutions converge to 0 in the distributional sense.
See  Theorem~\ref{THM:strong} below.

Previously, 
Albeverio, Haba,  and Russo  \cite{russo4}
studied a triviality issue for the two-dimensional SNLW:
\begin{equation}
\label{SNLW3}
 \dt^2 u - \Dl  u  + f(u) =  \xi, 
\end{equation}

\noi
where $f$ is a bounded smooth function.
Roughly speaking, they showed that 
solutions to~\eqref{SNLW3} with regularized noises
tend to that to the stochastic linear wave equation~\eqref{SNLW2}.
Let us point out several differences between \cite{russo4} and our current work
(besides considering the equations with/without damping).
Our argument is strongly motivated by the solution theory
recently developed in \cite{GKO}.
In particular, we carry out our analysis
in a natural solution space $C([0, T]; H^{-\eps}(\T^2))$, $\eps > 0$.
On the other hand, 
the analysis in~\cite{russo4} was carried out in the framework of Colombeau generalized functions,
and as such, their solution does not a priori belong to 
$C([0, T]; H^{-\eps}(\T^2))$.
Furthermore, the cubic nonlinearity $u^3$ 
does not belong to the class of nonlinearities considered in \cite{russo4}.
Regarding a triviality result
for the equation \eqref{SNLW3}, 
we also mention a recent work~\cite{ORSW}
on the stochastic wave equation with the sine nonlinearity.

In the parabolic setting, 
Hairer, Ryser,  and Weber \cite{HRW} 
studied 
 the following  stochastic Allen-Cahn equation on $\T^2$:
\begin{align}
 \dt u = \Dl u + u - u^3 + \al \xi. 
 \label{AC}
\end{align}

\noi
By suitably adapting 
 the strong solution theory due to Da Prato and Debussche~\cite{DPD},
they established  triviality for this equation;
(i) in the strong noise regime, 
  solutions to \eqref{AC}
with regularized noises  tend  to $0$ as the regularization is removed
and (ii) in the weak noise regime, 
  solutions to \eqref{AC} with regularized noises
converge to a solution to a deterministic nonlinear heat  equation.
We will establish analogues of these results
in the wave equation context; see Theorems \ref{THM:strong} and \ref{THM:weak}
below.

We also mention a recent work  \cite{OPT1}
by  Pocovnicu, Tzvetkov, and  the first author
on the cubic NLW on $\T^3$ with random initial data
of negative regularity.
As a byproduct of the well-posedness theory in this setting, 
they established a triviality result
for the defocusing cubic NLW (without renormalization) 
 with deterministic  initial data perturbed by rough random data.

Lastly, we point out that,  
in the context of nonlinear Schr\"odinger-type equations, 
instability results
 in negative Sobolev spaces,  analogous to triviality, 
  are known even in the deterministic setting;
see \cite{GO,OW}.
See also \cite{KapM, Cha} for analogous results
in the context of the modified KdV equation,
showing the necessity of renormalization
in the low regularity setting.

\subsection{Main results}

Given $N \in \N$, we denote  by $\P_N$ the Dirichlet projection onto the spatial frequencies 
$\Z^2_N \stackrel{\text{def}}= \{ |n| \le N \}$.
We study the following truncated equation:
\begin{equation}
\label{NLW2}
\dt^2 u_N - \Dl u_N + \dt u_N + u_N^3 = \al_N \xi_N
\end{equation}

\noi
with the truncated noise
\[\xi_N \stackrel{\textup{def}}{=} \P_N \xi.\]

\noi
Here, 
 $\{\al_N\}_{N \in \N}$ is a bounded sequence of non-zero real numbers, 
which reflects the strength of the noise.
Our goal is to study the asymptotic behavior of $u_N$ as $N \to \infty$
in the following two regimes:
\[ 
\text{(i) } \lim_{N \to \infty} \al_N^2 \log N = \infty
\qquad 
\text{and} \qquad 
\text{(ii) }
\lim_{N \to \infty} \al_N^2 \log N \in [0, \infty).\]

\noi
We refer to the case (i) (and the case (ii), respectively)
as the strong noise case (and the weak noise case, respectively).

Let us fix some notations.
We write $e_n(x) \stackrel{\textup{def}}{=} \frac1{2\pi}e^{in\cdot x}$, $n\in\Z^2$,  
for the orthonormal Fourier basis in $L^2(\T^2)$. 
Given $s \in \R$, we  define the  Sobolev space  $H^s (\T^2)$ by  the norm:
\[
\|f\|_{H^s(\T^2)} = \|\jb{n}^s\ft{f}(n)\|_{\l^2(\Z^2)},
\]

\noi
where $\ft{f}(n)$ is the Fourier coefficient of $f$ and  $\jb{\,\cdot\,} = (1+|\cdot|^2)^\frac{1}{2}$. 
We also set 
\begin{equation*}
\H^{s}(\T^2)  \stackrel{\textup{def}}{=} H^{s}(\T^2) \times H^{s-1}(\T^2).
\end{equation*}

\noi
When we work with space-time function spaces, 
we use short-hand notations such as
 $C_TH^s_x = C([0, T]; H^s(\T^2))$
 and $L^p_\o = L^p(\O)$.
Given $A, B \geq 0$, we also set $A\wedge B = \min(A, B)$.
\smallskip

\noi
{\bf (i) Strong noise case:}
 We first consider the strong noise case:
\begin{align}
\lim_{N \to \infty} \al_N^2 \log N = \infty.
\label{C1}
\end{align}
 
 \noi
 In this case, 
 the noise remains singular (in the limit),
 which provides a strong cancellation property of the solution $u_N$ to \eqref{NLW2}.

 Given $N \in \N$ and $\al_N\in \R$, 
 fix $\ld_N = \ld_N(\al_N) \geq 0$ (to be determined later; see \eqref{CN2} below). 
 We define 
a pair $(z_{0,N}^{\o}, z_{1,N}^{\o})$ of random functions
by the following random Fourier series:
\begin{equation} 
z_{0,N}^{\o} = \frac{\al_N}{\sqrt{2}}\sum_{|n| \le N} \frac{ g_n(\o)}{\jb{n}_N} e_n
 \qquad
\text{and}
\qquad
z_{1,N}^{\o} = \frac{\al_N}{\sqrt{2}}\sum_{|n| \le N}  h_n(\o) e_n,
\label{IV}
\end{equation}

\noi
where $\jb{n}_N$ is defined by 
\[\jb{n}_N\stackrel{\text{def}}{=} \sqrt{\ld_N+|n|^2}\]

\noi
and 
$\{ g_n \}_{n\in\Z^2}$ and $\{ h_n \}_{n\in\Z^2}$ are sequences of mutually independent standard complex-valued\footnote
{This means that $g_0,h_0\sim\NN_\R(0,1)$
and  
$\Re g_n, \Im g_n, \Re h_n, \Im h_n \sim \NN_\R(0,\tfrac12)$
for $n \ne 0$.}
 Gaussian random variables on 
a probability space $(\O,\F,P)$ conditioned so that $g_{-n} = \cj{g_n}$ and $h_{-n} = \cj{h_n}$, $n \in \Z^2$.
We also assume that 
$\{ g_n, h_n  \}_{n\in\Z^2}$ is independent of the space-time white noise $\xi$ in \eqref{SdNLW}.

We now state our main result.
Given $s, b \in \R$ and $T > 0$, 
we define the time restriction space
 \[H^{b}([0,T]; H^{s}(\T^2))\]

\noi
by   the norm
\begin{align}
\|u\|_{H^{b}([0, T]; H^s(\T^2))} 
= \inf\big\{\| v \|_{H^b(\R; H^s(\T^2))}:  v|_{[0, T]} = u \big\}.
\label{def1}
\end{align}

\noi
Here,  the $H^b(\R; H^s(\T^2))$-norm is defined by 
\[ \| v \|_{H^b(\R; H^s(\T^2))}
= \|\jb{\tau}^b\jb{n}^s \ft v(\tau, n)\|_{L^2_\tau \l^2_n}, 
\]

\noi
where $\ft v(\tau, n)$
denotes  the space-time Fourier transform of $v$.

\begin{theorem}\label{THM:strong}
Let  $\{\al_N\}_{N \in \N}$ 
be a  bounded sequence  of  non-zero real numbers, 
satisfying~\eqref{C1}.
Then, there exists a divergent sequence $\{\ld_N\}_{N \in \N}$
such that 
given  any $(v_0, v_1) \in \H^1(\T^2)$, 
 $T > 0$, $\eps >0$, and  $N \in \N$, 
there exists almost surely
  a unique solution $u_N
  \in C([0,T]; H^{-\eps}(\T^2))$
 to \eqref{NLW2} with initial data 
\begin{align}
(u_N,\dt u_N)|_{t = 0} = (v_0, v_1) + (z_{0,N}^{\o},z_{1,N}^{\o}),
\label{IV2}
\end{align}

\noi
where $(z_{0,N}^{\o},z_{1,N}^{\o})$ is as in \eqref{IV}.
Furthermore, 
 $u_N$ converges in probability to the trivial solution $u\equiv 0$ 
in 
$H^{-\eps}([0,T]; H^{-\eps}(\T^2))$ as $N\to \infty$.

\end{theorem}

Seeing the regularity of the stochastic term, one may think
that the natural space for the convergence is
$C([0,T]; H^{-\eps}(\T^2))$.
We need to work in a larger space
$H^{-\eps}([0,T]; H^{-\eps}(\T^2))$
in order to establish   convergence of the deterministic (modified) linear solution
(defined in \eqref{lin1} below).
See Lemma \ref{LEM:LWN}.

Our proof is strongly motivated by the arguments in \cite{HRW, OPT1}.
The main idea can be summarized as follows;
while we consider a model without renormalization, 
we artificially renormalize the nonlinearity
at the expense of modifying the linear operator.
More concretely, 
given a suitable choice of divergent constants $\ld_N$, 
we first rewrite the truncated equation \eqref{NLW2} as follows:
\begin{equation}
\label{NLW3}
\L_N u_N  + u_N^3 - \ld_N u_N = \al_N \xi_N , 
\end{equation}

\noi
where $\L_N$ denotes the modified damped wave operator:
\begin{align}
\L_N \stackrel{\textup{def}}{=} \partial_t^2 -\Delta + \dt +\ld_N.
\label{LN}
\end{align}

\noi
As we see below, the constant $\ld_N$ will play a role of a renormalization constant.
See \eqref{non1}.

We now set $\ld_N = \ld_N(\al_N)$ to be the unique solution  to 
\begin{equation} \label{CN2}
\ld_N = \frac{3\al_N^2}{8\pi^2} \sum_{|n| \le N} \frac{1}{\jb{n}_N^2} = \frac{3\al_N^2}{8\pi^2} \sum_{|n| \le N} \frac{1}{\ld_N+|n|^2}.
\end{equation}

\noi
 See Lemma \ref{LEM:CN} below.
With this choice of $\ld_N$, it
 is easy to see that 
 the corresponding  linear dynamics:
\begin{equation}\label{dSLWN}
\L_N  u_N = \al_N \xi_N
\end{equation}

\noi
possesses a unique invariant mean-zero Gaussian  measure $\mu_N$ on $\H^0(\T^2)$ with the covariance operator
\begin{align}
\frac{\al_N^2}2\begin{pmatrix}  \P_N (\ld_N-\Dl)^{-1} & 0 \\ 0 & 1 \end{pmatrix}.
\label{cov}
\end{align}

\noi
See 
Lemma~\ref{LEM:inv} below.
Our choice of random functions 
$(z_{0,N}^{\o}, z_{1,N}^{\o})$ in \eqref{IV}
is such that the random part of the initial data $(u_0, u_1)$ in \eqref{IV2}
is distributed by the Gaussian measure $\mu_N$.
We point out that by setting $\s_N$ by 
\begin{equation} \label{var}
\s_N \stackrel{\text{def}}= \E \big[(z_{0,N}^{\o}(x))^2\big] 
= \frac{\al_N^2}{8\pi^2}\sum_{|n| \le N} \frac{1}{\jb{n}_N^2},
\end{equation}

\noi
we have
\begin{align}
 \ld_N = 3\s_N.
 \label{ld1}
\end{align}

\noi
In Lemma \ref{LEM:CN} below, 
we show that 
\begin{align}
 \ld_N = \frac{3}{4\pi} \al_N^2 \log N + \text{lower order error},
\label{ld2}
\end{align}

\noi
which allows us to show that 
the sequence $\{(z_{0,N}^{\o}, z_{1,N}^{\o})\}_{N \in \N}$
is  almost surely uniformly bounded in $\H^{-\eps}(\T^2)$ for any $\eps>0$.

In the following, we describe an outline of the proof of Theorem \ref{THM:strong}.
The main idea is to apply 
the  Da Prato-Debussche trick \cite{DPD}
and  look for a solution to \eqref{NLW2} (or equivalently to~\eqref{NLW3}) of the form $u_N = z_N+v_N$, 
where $z_N$ denotes the singular stochastic part
and $v_N$ denotes a smoother residual part.

Given $N \in \N$, let  $z_N$ denote the solution to the linear equation \eqref{dSLWN}
with $(z_N, \dt z_N)|_{t = 0}
=  (z_{0,N}^{\o}, z_{1,N}^{\o})$.
It follows from the discussion above that $z_N$ is
a stationary process such that 
\[\mathrm{Law}\big((z_N(t), \dt z_N(t)) \big)  = \mu_N\]

\noi
for any $t \in \R_+$.
By expressing $z_N$
in the Duhamel formulation (= mild formulation),\footnote
{One can easily derive the propagator $\D_N(t)$ in \eqref{D1}
by writing   the linear damped wave equation
$ \L_N u = 0$ on the Fourier side
and solving it directly for each spatial frequency.
See \cite{IIOW} for the case $\ld_N = 0$ (on $\R^d$).
Then, 
a standard variation-of-parameter argument yields
the Duhamel formulation \eqref{eq:zN}.
} we have
\begin{align}\label{eq:zN}
z_N (t) =\dt\D_N(t)z_{0,N}^\o + \D_N(t)(z_{0,N}^\o+z_{1,N}^\o)+\alpha_N\int_0^t\D_N(t-t')
\P_N dW(t'),
\end{align}

\noi
where $\D_N(t)$ is given by 
\begin{align}
\D_N(t) \stackrel{\textup{def}}{=} e^{-\frac{t}{2}} \frac{\sin\Big( t \sqrt{\ld_N-\frac{1}{4}-\Delta}\Big)}{\sqrt{\ld_N-\frac{1}{4}-\Delta}}
\label{D1}
\end{align}

\noi
and  $W$ denotes a cylindrical Wiener process on $L^2(\T^2)$:
\begin{align}
W(t)=  \sum_{n \in \Z^2 } \beta_n (t) e_n.
\label{BM}
\end{align}

\noi
Here,  $\{ \beta_n \}_{n \in \Z^2}$ is a family of mutually independent complex-valued
Brownian motions conditioned so that $\be_{-n} = \cj{\be_n}$, $n \in \Z^2$.
Moreover, we assume that 
$\{ \beta_n \}_{n \in \Z^2}$ is independent from 
$\{ g_n, h_n  \}_{n\in\Z^2}$ in \eqref{IV}.
By convention, we  normalize $\be_n$ such that $\text{Var}(\be_n(t)) = t$.
Note that  the space-time white noise $\xi$ is given by  $\xi = \frac{\dd W}{\dd t}$.

By setting $v_N = u_N - z_N$, 
it follows from 
\eqref{NLW3} with \eqref{IV2} that $v_N$ satisfies the following equation:
\begin{equation}
\label{NLW4}
\begin{cases}
\L_N v_N  + (v_N+z_N)^3 - \ld_N (v_N+z_N) = 0 \\
(v_N,\dt v_N)|_{t = 0} = (v_0,v_1).
\end{cases}
\end{equation}

\noi
By invariance of the Gaussian measure $\mu_N$, 
we see that $z_N (t)$ has the same law as $z_{0,N}$ for any $t \in \R_+$.
In particular, 
it follows from \eqref{IV} that there is no uniform 
(in $N$) bound for $z_N(t)$, when measured in $L^2(\T^2)$.
This  causes an issue 
in studying 
the powers $z_N^2$ and $z_N^3$, uniformly in $N \in \N$.

In  \cite{GKO, GKOT}, it is at this point that we introduced
Wick renormalization and considered a renormalized equation to overcome this issue.
Our goal is, however, to study the limiting behavior
of the solution $u_N$ to \eqref{NLW2} {\it without} renormalization.
In our current problem, 
we overcome this difficulty
 by following the idea in~\cite{HRW, OPT1} 
and artificially introducing a renormalization constant
$\ld_N$ in \eqref{NLW3}.
By expanding the last two terms in \eqref{NLW4}, 
we have
\begin{align}
(v_N+z_N)^3 - \ld_N (v_N+z_N)
= v_N^3 + 3 v_N^2 z_N + 
3 v_N (z_N^2 - \s_N) +  (z_N^3 - 3\s_N z_N), 
\label{non1}
\end{align}

\noi
where we used  \eqref{ld1}.
Then, it follows from \eqref{var}
that the last two terms precisely correspond to the renormalized powers
of $z_N^2$ and $z_N^3$.
See  Section \ref{SEC:2} for further details.

This artificial introduction of renormalization as in \eqref{non1} allows
us to study the equation~\eqref{NLW4}
for $v_N$.
A standard contraction argument allows us to prove local well-posedness
of~\eqref{NLW4}, expressed in the Duhamel formulation:
\begin{align}
v_N(t) = 
v_N^\lin(t)
+\int_0^t\D_N(t-t') \NN(v_N + z_N)(t') d t', 
\label{non2}
\end{align}

\noi
where 
$\NN(v_N+z_N) = (v_N+z_N)^3 - \ld_N (v_N+z_N)$
and 
$v_N^\lin(t)$ denotes the linear solution
with deterministic initial data $(v_0, v_1)$:
\begin{align}
v_N^\lin(t) = 
\dt\D_N(t)v_{0} + \D_N(t)(v_{0}+v_{1}).
\label{lin1}
\end{align}

\noi
On the one hand, the diverging behavior \eqref{ld2} of $\ld_N$ 
and \eqref{D1}
allow us to show that 
 the  second term on the right-hand side of \eqref{non2}
tends to 0 as $N \to \infty$.
This explicit decay mechanism is analogous to that  in the parabolic case studied in \cite{HRW}.
On the other hand, 
the linear solution $v_N^\lin$
does not enjoy such a decay property in an obvious manner.
The crucial point here is that, in view of the asymptotics \eqref{ld2},  the 
modified linear operator $\L_N$ in \eqref{NLW3}
introduces a rapid oscillation
and,  as a result, 
$v_N^\lin$ tends to 0 as a space-time distribution.
This oscillatory nature of the problem 
is a distinctive feature of a dispersive problem, not present in the parabolic setting, 
and was also exploited in \cite{OPT1}.
In this paper, 
we go one step further.  
By exploiting the rapid oscillation in the form of oscillatory integrals, 
we show that 
$v_N^\lin$ tends to 0 
in $H^{-\eps}([0,T]; H^{1-\eps}(\T^2))$.
See Lemma~\ref{LEM:LWN}.
This essentially explains the proof of Theorem \ref{THM:strong}
for short times.

In order to prove the claimed convergence on an arbitrary time interval $[0, T]$, 
we need to establish a global-in-time control of the solutions $v_N$.
An energy bound in the spirit of Burq and Tzvetkov \cite{BT-JEMS}
allows us to  prove global existence of $v_N$.
Unfortunately, such an energy bound (at the level of $\H^1(\T^2)$)  
grows in $N$, which may cause a potential issue.
In general, it may be a cumbersome task
to obtain a global-in-time control on  $v_N$, uniformly in $N \in \N$.
One possible approach may be to adapt  the $I$-method 
argument employed in~\cite{GKOT}.
In our case,  however, the situation is much simpler since 
we know that the limiting solution is $u\equiv 0$, 
which allows us to reduce the problem to  a small data regime.

\begin{remark}\label{REM:a} \rm
(i) For simplicity, we only consider the regularization
via the Fourier truncation operator $\P_N$ in \eqref{NLW2}. 
By a slight modification of the proof, 
we can also  treat  regularization  by mollification with  a mollifier $\rho_{\eps}$, $\eps\in (0,1]$ and taking the limit $\eps\rightarrow0$.

\smallskip

\noi
(ii) We consider the stochastic NLW with damping.
This allows us to have an invariant Gaussian measure $\mu_N$
for the linear dynamics \eqref{dSLWN}, 
which in turn implies that the renormalization constant $\ld_N$
defined in \eqref{CN2} and \eqref{ld1}
is time independent.
If we consider the stochastic NLW without damping, 
then $\ld_N$ would be time dependent.
This would then imply that 
  the modified linear operator $\L_N$
in \eqref{LN} is with a variable coefficient $\ld_N(t)$,
introducing an extra complication to the problem.
This is the reason we chose to study  the stochastic NLW with damping.

\smallskip

\noi
(iii) 
In the parabolic setting \cite{HRW}, the triviality result was stated 
only with deterministic initial data.
Namely, there was no need to add the random initial data
as 
 in~\eqref{IV2}.
In \cite{HRW}, 
 the residual part $v_N$
 satisfies an analogue of 
\eqref{NLW4}
with initial data
essentially of the form (written in the wave context):
\begin{align}
(v_N, \dt v_N)|_{t = 0} = (v_0, v_1) - (z_{0,N}^{\o},z_{1,N}^{\o}).
\label{IV4}
\end{align}

\noi
See the equation $(\Phi^{aux}_\eps)$
on p.\,6 in \cite{HRW}.
In the parabolic setting, this does not cause any difficulty
since the strong parabolic smoothing allows us to handle
rough initial data of the form \eqref{IV4} in the deterministic manner.
On the other hand, 
in the current wave context, 
we can not handle 
the random data in \eqref{IV4}, 
unless we introduce a further renormalization
(which would violate the point of this paper).

Let us  now try to see what happens if we  directly work with the non-stationary solution
to the linear equation \eqref{dSLWN}.
Let $\wt z_N$ denote the solution to \eqref{dSLWN}
with the trivial (= zero) initial data, given by 
\begin{align}
\wt z_N (t) =\alpha_N\int_0^t\D_N(t-t')
\P_N dW(t')
\label{XX1}
\end{align}

\noi
Then, by setting
\[ \wt \ld_N(t) = 3 \wt \s_N(t) = 3 \E\big[(\wt z_N(t, x))^2\big], \]

\noi
we can rewrite  \eqref{NLW2} as
\begin{equation}
\L_N u _N  + \big( u_N^3 - \wt \ld_N(t) u_N\big)
- (\ld_N - \wt \ld_N(t)) u_N = 0, 
\label{XX1a}
\end{equation}

\noi
where $\L_N$ is as in \eqref{LN}.
By writing $u_N = v_N + \wt z_N$, 
we easily see that 
the expression $u_N^3 - \wt \ld_N(t) u_N$
can be treated as in  \eqref{non1}
without causing any difficulty.
On the other hand, 
the weaker smoothing property 
of the damped wave equation (as compared to the heat equation)
causes an issue in treating
the last term $(\ld_N - \wt \ld_N(t)) u_N$. 
Define $z^{\text{hom}}_N $ by 
\begin{align}
z^{\text{hom}}_N (t) =\dt\D_N(t)z_{0,N}^\o + \D_N(t)(z_{0,N}^\o+z_{1,N}^\o), 
\label{XX2}
\end{align}

\noi
where 
$(z_{0,N}^\o, z_{1,N}^\o)$ is as in \eqref{IV}.
Then, 
it follows
from 
\eqref{eq:zN}, 
\eqref{XX1}, and \eqref{XX2}
together with  
 independence of 
$\{ g_n, h_n  \}_{n\in\Z^2}$ in~\eqref{IV}
and the cylindrical Wiener process $W$ in~\eqref{BM}
that 
\begin{align*}
 \ld_N - \wt \ld_N(t)  
 & = 3\E\big[(z^{\text{hom}}_N (t, x))^2\big]\\
& = \frac{3 e^{-t}\al_N^2}{8\pi^2} \sum_{|n|\leq N}
\bigg\{\frac{1}{ \jb{n}_N^2}
\bigg(\cos\Big(t \sqrt{\ld_N-\tfrac{1}{4}+|n|^2 }\Big)\bigg)^2\\
& \hphantom{XXXXXXXX}
+ 
\frac{\bigg(\sin\Big( t \sqrt{\ld_N-\frac{1}{4}+|n|^2}\Big)\bigg)^2}
{\ld_N-\frac{1}{4}+|n|^2}\bigg\}
 + O(1)\\
& = \frac{3 e^{-t}\al_N^2}{8\pi^2} \sum_{|n|\leq N}
\frac{1}
{\ld_N-\frac{1}{4}+|n|^2}
 + O(1), 
\end{align*}

\noi
which is logarithmically divergent for any $t \geq 0$.
This shows that the last term in~\eqref{XX1a} 
(under the Duhamel integral) can not be treated 
uniformly in $N \in \mathbb{N}$, thus exhibiting
non-trivial difficulty  in the non-stationary case.
Compare this with the heat case,  where the corresponding expression
for  $\ld_N - \wt \ld_N(t)$
is uniformly bounded in $N \in \mathbb{N}$ for any $t > 0$ 
(and logarithmically divergent when  $t = 0$) thanks to the strong smoothing property.

%
%
%
%
%
%
%

\smallskip

\noi
(iv) In Theorem \ref{THM:strong}, 
we treated the cubic case. It would be of interest to investigate the issue of triviality
for a higher order nonlinearity.
See also Remarks \ref{REM:x} and \ref{REM:higher}
on this issue in the weak noise case.

Our argument also makes use of the defocusing nature of the equation
in an essential manner.
In the focusing case, the modified linear operator $\L_N$ in \eqref{LN}
would be 
$\L_N = \partial_t^2 -\Delta + \dt -\ld_N$.
Namely, the diverging constant $\ld_N$ appears with a wrong sign
and we do not know how to proceed at this point.

\end{remark}

\smallskip

\noi
{\bf (ii) Weak noise case:}
Next, we consider the weak noise case:
\begin{align}
\lim_{N \to \infty} \al_N^2 \log N = \kk^2 \in [0, \infty).
\label{C2}
\end{align}
 
 \noi
 In particular,  we have $\al_N \to 0$ and thus 
 we expect convergence to a deterministic damped NLW.
 In this case, 
 we set 
 \begin{align}
 \L \stackrel{\text{def}}= \partial_t^2 -\Delta + \dt +1.
 \label{C2a}
 \end{align}
 
\noi
Namely, we can simply set $\ld_N \equiv 1$ in the previous discussion.
With a slight abuse of notation, we then 
define $\mu_N$ to be the mean-zero Gaussian measure
 on $\H^0(\T^2)$ with the covariance operator
 \begin{align*}
\frac{\al_N^2}2\begin{pmatrix}  \P_N (1-\Dl)^{-1} & 0 \\ 0 & 1 \end{pmatrix}.
\end{align*}

\noi
Then,  it follows that $\mu_N$ is 
 the unique invariant measure 
 for the  linear equation:
\begin{equation}\label{dSLWN2}
\L u_N = \al_N \xi_N.
\end{equation}

\noi
With a slight abuse of notation, 
we  use  $z_N$ to denote the solution to \eqref{dSLWN2} 
with the random initial data 
$(z_N, \dt z_N)|_{t = 0} = (z_{0,N}^{\o}, z_{1,N}^{\o})$
distributed by $\mu_N$
as in the previous case.
In particular, 
the random initial data in this case is given by 
 \eqref{IV} with $\ld_N = 1$, namely
\begin{equation} 
z_{0,N}^{\o} = \frac{\al_N}{\sqrt{2}}\sum_{|n| \le N} \frac{ g_n(\o)}{\jb{n}} e_n
 \quad
\text{and}
\quad
z_{1,N}^{\o} = \frac{\al_N}{\sqrt{2}}\sum_{|n| \le N}  h_n(\o) e_n.
\label{IV3}
\end{equation}

We now state our second result.

\begin{theorem}\label{THM:weak}
Let  $\{\al_N\}_{N \in \N}$ 
be a  bounded sequence  of real numbers, 
satisfying \eqref{C2}
for some $\kk^2 \in [0,  \infty)$ .
Then, given  any $(v_0, v_1) \in \H^1(\T^2)$,  $T>0$, $\eps > 0$, and  $N\in\N$,
there exists
almost surely   a unique solution $u_N
\in C([0, T]; H^{-\eps}(\T^2))$
 to \eqref{NLW2} with initial data 
\begin{align}
(u_N,\dt u_N)|_{t = 0} = (v_0, v_1) + (z_{0,N}^{\o},z_{1,N}^{\o}),
\label{IV5}
\end{align}

\noi
where $(z_{0,N}^{\o},z_{1,N}^{\o})$ is as in \eqref{IV3}.
Furthermore, 
 $u_N$ converges in probability to 
 $w_{\kk}$
in 
$C([0, T]; H^{-\eps}(\T^2))$
as $N\to \infty$, 
where $w_{\kk}$ is the unique solution to the following deterministic
damped NLW:
\begin{equation}
\label{dNLWw}
\begin{cases}
 \dt^2 w_{\kk} - \Dl w_{\kk} + \dt w_{\kk} + \frac3{4\pi} \kk^2 w_{\kk} + w_{\kk}^3 = 0 \\
 (w_{\kk},\dt w_{\kk})|_{t = 0} = (v_0,v_1).
\end{cases}
\end{equation}
\end{theorem}

Recall that, in Theorem \ref{THM:strong}, 
we needed to study the convergence in a space larger than
$C([0, T]; H^{-\eps}(\T^2))$.
This was due to 
the convergence property of the deterministic (modified) linear solution
$v_N^\lin$
in \eqref{lin1}.
In Theorem \ref{THM:weak}, 
we estimate the difference of the solution $u_N$ to \eqref{NLW2}
with initial data~\eqref{IV5}
and the limiting solution $w_\kk$ to \eqref{dNLWw}.
As such, the deterministic part $(v_0, v_1)$ of the initial data  cancels each other,
allowing us to prove 
the convergence  in a natural
space
$C([0, T]; H^{-\eps}(\T^2))$.

\begin{remark}\rm
As mentioned in Remark \ref{REM:a}, 
we consider the equation with damping
so that the linear equation \eqref{dSLWN2} preserves the Gaussian measure $\mu_N$.
This naturally yields the damped equation \eqref{dNLWw} 
as  the  limiting deterministic equation.
In this weak noise regime, 
however, 
 it is possible to introduce another parameter $\wt \al_N$ and tune the parameters 
 such that  the dynamics  converges to that generated by a standard deterministic NLW without damping.

Consider the following SdNLW:
\begin{equation}
\label{NLW6}
\dt^2 u_N - \Dl u_N + \wt{\al}_N \dt u_N + u_N^3 = \al_N \xi_N,
\end{equation}

\noi
where $\wt \al_N $ is a positive number,  tending to 0 as $N \to \infty$.
For $N \in \N$, set $\g_N^2 = \frac{\al_N^2}{2\wt{\al}_N}$.
We assume that $\{\g_N^2 \}_{N \in \N}$ is bounded.\footnote
{This in particular implies that $\al_N$ tends to $0$
since our assumption states that $\wt \al_N$ tends to $0$.}
Then, by repeating the proof of Theorems~\ref{THM:strong} and~\ref{THM:weak}, 
it is straightforward to see that the limiting behavior
of the solution $u_N$ to \eqref{NLW6}
is determined by 
\[ \lim_{N \to \infty} \g_N^2\log N = 
\lim_{N \to \infty} \frac{\al_N^2}{2\wt{\al}_N}
 \log N = \g^2\in [0, \infty].\]

\noi
We have the following two scenarios.
(i) If $\g^2 = \infty$, then the solution $u_N$ to \eqref{NLW6}  converges  to $0$.
(ii)
If $\g^2 \in [0,\infty)$, 
then the solution $u_N$ to \eqref{NLW6} converges  to the solution $w_\g$,
satisfying the following deterministic NLW (i.e.~without damping):
\[
\dt^2 w_{\g} - \Dl w_{\g} + \frac3{4\pi} \g^2 w_{\g} + w_{\g}^3 = 0.
\]

The main point is that the tuning of the parameters, 
making the sequence $\{\g_N^2\}_{N\in \N}$ bounded, 
allows us to make use of certain invariant Gaussian measures
for the (modified) linear dynamics.

\end{remark}

\begin{remark}\label{REM:x}\rm
In the weak noise case, it is possible to adapt our argument to 
a general defocusing power-type nonlinearity $u^{2k+1}$.
See Remark~\ref{REM:higher} for further details.

\end{remark}

\section{Preliminary results for the strong noise case}
\label{SEC:2}

In this section, we go over some preliminary materials
for the strong noise case (Theorem~\ref{THM:strong}), whose proof is presented in Section \ref{SEC:tri}.
In Subsection~\ref{SUBSEC:lin}, we prove that 
the Gaussian measure $\mu_N$ with the covariance operator
\eqref{cov}
is  the (unique) invariant measure for the 
linear  stochastic  wave equation \eqref{dSLWN}.
In Subsection \ref{SUBSEC:ld}, 
we establish the asymptotic behavior \eqref{ld2} of
the renormalization constant $\ld_N$.
In Subsection~\ref{SUBSEC:wick}, we define the renormalized powers $:\! z_N^\l\!:$
for the solution  $z_N$  to the linear equation~\eqref{dSLWN}
with $(z_N, \dt z_N)|_{t = 0}
=  (z_{0,N}^{\o}, z_{1,N}^{\o})$.
Lastly, 
in Subsection~\ref{SUBSEC:lin2}, 
we study the decay property of the deterministic linear solution 
 $v_N^{\lin}$ defined in \eqref{lin1}.

\subsection{On the invariant measure for the linear equation}
\label{SUBSEC:lin}
We begin by describing the invariant measure for the linear stochastic equation \eqref{dSLWN}:
\begin{equation*}
\L_N  u_N = \al_N \xi_N.
\end{equation*}

\noi
 We only sketch a proof since the argument is classical; see, for example, \cite{GKOT,ORTz} for a more detailed discussion.

\begin{lemma}\label{LEM:inv}
The linear stochastic wave equation \eqref{dSLWN} possesses a \textup{(}unique\textup{)} invariant mean-zero Gaussian  measure $\mu_N$ on $\H^0(\T^2)$ with the covariance operator
given in \eqref{cov}.
\end{lemma}

\begin{proof} 
We only present a sketch of the proof.
For $|n|\leq N$, 
let 
\[ X_n = \begin{pmatrix}
\ft{u}_N(n)\\ \dt\ft{u}_N(n)
\end{pmatrix}.\] 

\noi
Then, in view of 
\eqref{BM}, 
we can rewrite 
the linear equation \eqref{dSLWN} as the following system of stochastic differential equations:
\begin{align}
dX_n = \begin{pmatrix}
0&1\\ -\jb{n}_N^2&0
\end{pmatrix}X_ndt +\bigg[\begin{pmatrix}
0&0\\0&-1
\end{pmatrix}X_ndt+\begin{pmatrix}
0\\\al_Nd\beta_n
\end{pmatrix}\bigg].
\label{sys}
\end{align}

\noi
The first part on the right-hand side
corresponds to the (modified) linear wave equation 
(without damping) whose semi-group acts as a rotation on 
each component of the vector~$X_n$. Since the distribution of a complex-valued Gaussian random variable is invariant under a rotation, 
we see that the solution to this linear wave equation, starting from the random initial data 
$(z_{0,N}^{\o},z_{1,N}^{\o})$ in \eqref{IV}, is stationary.

The second part on the right-hand side of \eqref{sys} corresponds to the Langevin equation for the velocity $\dt \ft{u}_N(n)$:
\[ d(\dt \ft{u}_N(n)) = -(\dt \ft{u}_N(n))dt + \alpha_Nd\beta_n,\]

\noi
whose solution is given by a complex-valued Ornstein-Uhlenbeck process.
Namely, 
its real and imaginary parts are given by independent
 Ornstein-Uhlenbeck processes.
Hence, it  has a unique invariant measure given by the Gaussian distribution 
$\NN_\C (0,\frac{\alpha_N^2}{2})$ (see, for example,~\cite[Theorem~7.4.7]{Kuo}), which is precisely the law of $\ft{z}_{1,N}^\o = \dt\ft{z}_N(0)$ defined in~\eqref{IV}.

For each $n \in \Z^2$ with $|n|\leq N$, 
 the generator of the dynamics \eqref{sys}
is given by the sum of the generators
of the first and second parts on the right-hand side of \eqref{sys}.
Hence, 
we conclude that 
the full linear stochastic wave equation  \eqref{dSLWN}, starting from 
$(z_{0,N}^{\o},z_{1,N}^{\o})$ in~\eqref{IV}, is also stationary. 
This means that the mean-zero Gaussian measure $\mu_N$ with the covariance operator \eqref{cov}
is invariant under \eqref{dSLWN}. 
One can also prove that $\mu_N$ is actually the unique invariant measure for this equation;
see Theorems 11.17 and 11.20 in \cite{DPZ14}.
\end{proof}

Recall that  $z_N$ defined by \eqref{eq:zN}  satisfies the linear stochastic 
 wave equation~\eqref{dSLWN}.
 Then, due to the invariance of  $\mu_N$  under the flow of \eqref{dSLWN}, the variance of $z_N(t)$ is 
 time independent and given by \eqref{var}:
\begin{align}
\s_N = \E \big[ (z_N(t,x))^2\big] = \E \big[ (z_N(0,x))^2\big] = \frac{\al_N^2}{8\pi^2}\sum_{|n| \le N} \frac{1}{\jb{n}_N^2}.
\label{var2}
\end{align}

\subsection{On the renormalization constant}
\label{SUBSEC:ld}

In this subsection, we study asymptotic properties
of the renormalization  constant $\ld_N$ implicitly defined  by \eqref{CN2}:
\begin{equation}
\ld_N = \frac{3\al_N^2}{8\pi^2} \sum_{|n| \le N} \frac{1}{\jb{n}_N^2}
= \frac{3\al_N^2}{8\pi^2} \sum_{|n| \le N} \frac{1}{\ld_N + |n|^2}.
\label{CN3}
\end{equation}

\noi
In particular, we prove the following lemma
on the asymptotic behavior of $\ld_N$ as $N \to \infty$.
See Lemma 3.1 in \cite{HRW} and  Lemma 6.1 in \cite{OPT1}
for analogous results.

\begin{lemma}\label{LEM:CN}
Given  $N \in \N$, there exists a unique number $\ld_N > 0$ satisfying the equation~\eqref{CN3}.
Moreover, if 
 $\{\al_N\}_{N \in \N}$ is a bounded sequence of non-zero real numbers
such that  $\lim_{N \to \infty} \al_N^2 \log N= \infty$, then we have 
\begin{align}
\ld_N = \frac3{4\pi} \al_N^2 \log N + O( \al_N^2 \log \log N)
\label{CN4}
\end{align}

\noi
as $N \to \infty$.

\end{lemma}

Before proceeding to the  proof of Lemma \ref{LEM:CN}, 
we first recall the following bound.
See  Lemma~3.2 in~\cite{HRW}.

\begin{lemma} \label{LEM:log}
Let $a$, $N \ge 1$.
Then, we have
\[
\bigg| \sum_{|n|\leq N} \frac{1}{a+|n|^2} - \pi \log \bigg( 1+\frac{N^2}{a} \bigg) \bigg|
\les \frac{1}{\sqrt{a}} \min \left( 1, \frac{N}{\sqrt{a}} \right).
\]
\end{lemma}

We now present a proof of  Lemma \ref{LEM:CN}.

\begin{proof}[Proof of Lemma \ref{LEM:CN}]
Given $N \in \N$, let $\ld_N$ be as in \eqref{CN3}.
As $\ld_N$ increases  from $0$ to $\infty$, 
the right-hand side of \eqref{CN3} decreases from $\infty$  to $0$.
Hence, for each $N \in \N$, there exists a unique solution $\ld_N > 0$ to~\eqref{CN2}. 

From $\ld_N > 0$, we obtain
an upper bound $\ld_N\les \alpha_N^2\log N$.
From this upper bound
and the uniform boundedness of $\al_N$, 
we also obtain a  lower bound
$\ld_N\ges \alpha_N^2\log N$
for any sufficiently large $N \gg 1$.
Hence,  we have
\begin{align}
\ld_N\sim\alpha_N^2\log N
\label{CN4a}
\end{align}

\noi
for any $N \gg 1$.

From  Lemma \ref{LEM:log},
we have  
\[\sum_{|n| \le N} \frac{1}{\jb{n}^2} = 2 \pi \log N + O(1).\]

\noi
Then, in view of the uniform boundedness of $\al_N$, 
 the error term $R_N$  is given by 
\begin{equation}
\begin{split} R_N 
 & 
 = \ld_N -  \frac3{4\pi} \al_N^2 \log N 
= \ld_N - \frac{3}{8\pi^2}\al_N^2 \sum_{|n| \le N} \frac{1}{\jb{n}^2}
+ O(\al_N^2)\\
& = \frac{3}{8\pi^2} \al_N^2 \sum_{|n|\leq N } \bigg( \frac{1}{\ld_N + |n|^2} - \frac{1}{\jb{n}^2}\bigg)
+ O(\al_N^2)
\\
& = \frac{3}{8\pi^2} \al_N^2 \sum_{|n|\leq N } \frac{1-\ld_N}{(\ld_N + |n|^2)\jb{n}^2}
+ O(\al_N^2).
\end{split}
\label{CN5}
\end{equation}

\noi
Using \eqref{CN4a}, we can estimate 
the contribution to $R_N$ in \eqref{CN5}
from $\big\{| n |\ges |\alpha_N|\sqrt{\log N}\big\}$
as $O(\al_N^2)$, 
while the  contribution to $R_N$ in \eqref{CN5}
from $\big\{| n| \ll  |\alpha_N|\sqrt{\log N}\big\}$
is $O(\al_N^2 \log \log N)$.
Putting everything together, we obtain \eqref{CN4}.
\end{proof}

\subsection{On the Wick  powers}
\label{SUBSEC:wick}

Given $N \in \N$, let $z_N$ be the solution to the linear equation~\eqref{dSLWN}
with $(z_N, \dt z_N)|_{t = 0}
=  (z_{0,N}^{\o}, z_{1,N}^{\o})$.
In the following, we define the renormalized powers of $z_N$
and establish their regularity and  decay properties.

Recall that the Hermite polynomials $H_k(x; \s)$
are defined via the generating function:
\begin{equation*}
F(t, x; \s)  =  e^{tx - \frac{1}{2}\s t^2} = \sum_{k = 0}^\infty \frac{t^k}{k!} H_k(x;\s).
 \end{equation*}
	
\noi
In the following, we list the first few Hermite polynomials
for readers' convenience:
\begin{align}
& H_0(x; \s) = 1, 
\quad 
H_1(x; \s) = x, 
\quad
H_2(x; \s) = x^2 - \s, 
\quad  H_3(x; \s) = x^3 - 3\s x.
\label{Herm}
\end{align}

\noi
Then, 
given $\l \in  \Z_{\geq 0} \stackrel{\text{def}}{=} \N \cup\{0\}$,
we  define the Wick powers $:\!  z_N^\l  \!:$
by 
\begin{align}
:\!  z_N^\l (t,x)\!:  \, = H_\l( z_N(t,x); \s_N)
\label{Wick1}
\end{align}
in a pointwise manner, 
where $\s_N$ is as in \eqref{var2}.

Before proceeding further,  
let us first state several lemmas.
The first lemma states
the orthogonality  property of Wick products \cite[Theorem I.3]{Simon}.
See also \cite[Lemma 1.1.1]{Nualart}.

\begin{lemma}\label{LEM:W1}
Let $f$ and $g$ be Gaussian random variables with variances $\s_f$ and $\s_g$.
Then, we have 
\begin{align*}
\E\big[ H_k(f; \s_f) H_m(g; \s_g)\big] = \dl_{km} k! \big\{\E[ f g] \big\}^k.
\end{align*}

\noi
Here, $\dl_{km}$ denotes the Kronecker delta function.

\end{lemma}


Next, we recall the Wiener chaos estimate.
Let $\{ g_n \}_{n \in \N}$ be a sequence of independent standard Gaussian random variables defined on a probability space $(\O, \mathcal{F}, P)$, where $\mathcal{F}$ is the $\s$-algebra generated by this sequence. 
Given $k \in \Z_{\geq 0}$, 
we define the homogeneous Wiener chaoses $\mathcal{H}_k$ 
to be the closure (under $L^2(\O)$) of the span of  Fourier-Hermite polynomials $\prod_{n = 1}^\infty H_{k_n} (g_n)$, 
where
$H_j$ is the Hermite polynomial of degree $j$ and $k = \sum_{n = 1}^\infty k_n$.
(This implies
that $k_n = 0$ except for finitely many $n$'s.)
Then,  we have the following classical Wiener chaos estimate; see \cite[Theorem I.22]{Simon}.
\begin{lemma}\label{LEM:hyp}
Let $\l \in \Z_{\geq 0}$.
Then, we have 
 \begin{equation*}
\| X\|_{L^p(\O)}  \leq (p-1)^\frac{\l}{2}\|X\|_{L^2(\O)}
 \end{equation*}

\noi
for any random variable $X \in \mathcal{H}_\l$ and  any finite $p \geq 1$.

\end{lemma}

Our main goal here is to prove the following  regularity and decay properties
of the Wick powers $:\!  z_N^\l  \!:$.

\begin{proposition}\label{PROP:estzN}
\textup{(i)} 
Let $\l \in \N$.
Then, given any finite  $p,q \geq 1$,  $T>0$, and $\eps>0$,
we have
\[
\lim_{N \to \infty} \E \Big[ \| :\! z_N^\l (t)\!: \|_{L^q_T W_x^{-\eps,\infty}}^p \Big] =0.
\]

\smallskip

\noi
\textup{(ii)}
Given any finite  $p \geq 1$,  $T>0$, and $\eps>0$,
 we have
\begin{align}
\lim_{N \to \infty} \E \Big[\| z_N\|_{C_TH_x^{-\eps}}^p \Big] =0.
\label{estZ1}
\end{align}

\end{proposition}

\begin{proof}
(i) By Sobolev's inequality,  it suffices to show that
\begin{equation}\label{estzN1}
\lim_{N \to \infty} \E \Big[ \| :\! z_N^\l (t)\!: \|_{L^q_T W_x^{-\eps,r}}^p \Big] =0
\end{equation}

\noi
for any small $\eps>0$ and sufficiently large $r \gg1$.
We follow the argument in the proof of Proposition~2.1 in \cite{GKO}.
Fix $t \in \R_+$ and $x, y \in \T^2$.
Then, by Lemma \ref{LEM:W1} with the invariance of the distribution of $z_N(t)$
and \eqref{IV}, we have 
\begin{align*}
\E \big[ :\! z_N^\l (t,x)\!: :\! z_N^\l (t,y)\!: \big]
&= \l! \, \E [ z_N (t,x) z_N (t,y) ]^{\ell}
= \frac{\l !}{(8 \pi^2)^\l}   \Bigg\{ \sum_{|n| \le N} \frac{\al_N^2}{\jb{n}_N^2} e^{in\cdot (x-y)} \Bigg\}^{\l} \\
&= \frac{\l !}{(8 \pi^2)^\l} \sum_{n_1, \dots, n_{\l} \in \Z^2_N} 
\bigg( \prod_{j=1}^{\l} \frac{\al_N^2}{\jb{n_j}_N^2} \bigg) e^{i(n_1+\cdots+n_{\l})\cdot (x-y)}.
\end{align*}

\noi
By applying the Bessel potentials $\jb{\nb_x}^{-\eps}$
and $\jb{\nb_y}^{-\eps}$ of order $\eps$ and then setting $x = y$, 
we obtain
\begin{align*}
\E \Big[ |\jb{\nb}^{-\eps} :\! z_N^\l (t,x)\!:|^2 \Big]
&\sim  \sum_{n_1, \dots, n_{\l} \in \Z^2_N} \bigg( \prod_{j=1}^{\l} \frac{\al_N^2}{\jb{n_j}_N^2} \bigg) \jb{n_1+\cdots+n_{\l}}^{-2\eps}\\
& \les 
\ld_N^{-\frac \eps2}
 \sum_{n_1, \dots, n_{\l} \in \Z^2_N}   \frac{1}{
\prod_{j=1}^{\l-1}\jb{n_j}^2 \cdot \jb{n_\l}^{2-\eps} \jb{n_1+\cdots+n_{\l}}^{2\eps}}\\
& \les \ld_N^{-\frac \eps2},
\end{align*}

\noi
uniformly for all sufficiently large $N \gg 1$, 
where, in the first inequality, 
we used the uniform boundedness of  $\al_N$,
 Lemma \ref{LEM:CN},  
and the bound $\jb{n}_N\geq \ld_N^{\frac \eps2}\jb{n}^{1-\eps}$
for $\eps \in [0, 1]$.
Then, from Minkowski's integral inequality and the Wiener chaos estimate (Lemma \ref{LEM:hyp}), we obtain
\begin{equation}
\begin{split}
\Big\| \| :\! z_N^\l \!: \|_{L^q_T W_x^{-\eps,r}} \Big\|_{L^p(\O)}
&\le \Big\| \| \jb{\nb}^{-\eps} :\! z_N^\l (t, x)\!: \|_{L^p(\O)} \Big\|_{L^q_T L_x^r} \\
&\le p^{\frac{\l}{2}} \Big\| \| \jb{\nb}^{-\eps} :\! z_N^\l (t, x)\!: \|_{L^2(\O)} \Big\|_{L^q_T L_x^r}\\
& \le  C_{ \l}\,p^{\frac{\l}{2}} T^{\frac{1}{q}} \ld_N^{-\frac{\eps}{4}} 
\end{split}
\label{tail-z}
\end{equation}

\noi
for any finite $p \ge \max (q,r)$.
The claim \eqref{estzN1} follows from \eqref{tail-z} and 
the asymptotic behavior \eqref{CN4} 
of $\ld_N$ proved in  Lemma \ref{LEM:CN}.

\smallskip

\noi (ii)
We prove \eqref{estZ1} for any small $\eps > 0$.
Given $N \in \N$, define $\Psi_N$ by 
\[
\Psi_N(t) = \int_0^t \D_N(t-t') \P_N dW(t'),
\]

\noi
where $\D_N$ and $W$ are as in \eqref{D1} and \eqref{BM}.
With a slight abuse of notation, 
define a Fourier multiplier operator 
 $D_N$ by the following symbol 
 \begin{equation}\label{Dn}
D_N(n) = \sqrt{\lambda_N-\frac14+|n|^2}.
\end{equation}

\noi
Then, it follows from 
 \eqref{eq:zN}, the unitarity of $e^{\pm itD_N}$ on $H^{s}(\T^2)$, Minkowski's integral inequality, 
 and Lemma \ref{LEM:hyp} that 
\begin{align*}
\Big\|\|  z_N  & - \al_N \Psi_N \|_{C_TH_x^{-\eps}} \Big\|_{L^p(\O)}
\les \Big\| \| z_{0,N}^{\o} \|_{H^{-\eps}} \Big\|_{L^p(\O)} 
+ \Big\| \| D_N^{-1} z_{1,N}^{\o} \|_{H^{-\eps}} \Big\|_{L^p(\O)} \\
&\les \Big\| \| \jb{\nb}^{-\eps} z_{0,N}^{\o}(x) \|_{L^p(\O)} \Big\|_{L^2_x}  
+ \Big\| \| D_N^{-1} \jb{\nb}^{-\eps} z_{1,N}^{\o} (x)\| _{L^p(\O)} \Big\|_{L^2_x} \\
&\les_p \bigg( \sum_{|n| \le N} \frac{\al_N^2}{\jb{n}^{2\eps} \jb{n}_N^2} \bigg)^{\frac{1}{2}} 
+ \bigg( \sum_{|n| \le N} \frac{\al_N^2}{\jb{n}^{2\eps} \big(D_N(n)\big)^2} \bigg)^{\frac{1}{2}}
\les \ld_N^{-\frac{\eps}{4}}.
\end{align*}

\noi
In the last step, we once again used the uniform boundedness of $\al_N$
and also the following bound:
\begin{equation}\label{eq:Dn}
D_N(n)\sim\jb{n}_N\gtrsim\ld_N^{\frac \eps2}\jb{n}^{1-\eps}
\end{equation}

\noi
uniformly for all sufficiently large $N \gg 1$, 
 in view of \eqref{Dn} and Lemma~\ref{LEM:CN}.

Hence, it suffices to show that
\begin{equation} \label{stcon}
\bigg\| \sup_{0 \le t \le T} \| \Psi_N(t) \|_{H^{-\eps}} \bigg\|_{L^p(\O)}
\les \ld_N^{-\frac{\eps}{4}}.
\end{equation}

\noi
In view of  \eqref{eq:Dn}, 
one can easily modify 
the proof of Proposition 2.1 in \cite{GKO} to obtain \eqref{stcon}.
In the following, however, 
we apply the factorization method 
based on the elementary identity:
\begin{align}
\int_{t_2}^t (t-t_1)^{\g-1} (t_1-t_2)^{-\g} dt_1 = \frac{\pi}{\sin \pi \g}
\label{fa1}
\end{align}

\noi
for any  $\g \in (0,1)$ and $0 \leq t_2 \leq t$;
see \cite[Section 5.3]{DPZ14}.

Recall from \eqref{D1} and \eqref{Dn} that
\begin{align}
\D_N(t) = e^{-\frac t2} \frac{\sin(tD_N)}{D_N}. 
\label{D2}
\end{align}

\noi
Together with \eqref{fa1}, 
we have
\begin{align}
\begin{split}
\Psi_N(t) 
& = \sum_{\s\in \{-1, 1\}}\s
\int_0^t \frac{ e^{i \s (t- t') D_N}}{2i D_N}  \P_N dW(t')\\
& = \frac{\sin \pi \g}{\pi}
\sum_{\s\in \{-1, 1\}}\s
\int_0^t \frac{ e^{i \s (t- t_1) D_N}}{2i D_N}  (t - t_1)^{\g - 1} Y_{\s, N}(t_1) dt_1, 
\end{split}
\label{fa2}
\end{align}

\noi
where
\[ Y_{\s, N}(t_1) = \int_0^{t_1}  e^{i \s (t_1- t_2) D_N}  (t_1 - t_2)^{-\g}\P_N dW(t_2).\]

%
%
%
\noi
Then, from \eqref{fa2} and the boundedness of 
$e^{i \s t D_N}$ on $H^s(\T^2)$, 
we have 
\begin{align}
\|\Psi_N\|_{L^p_\o L^{\infty}_t([0,T])H_x^{-\eps}} 
&\les 
\sum_{\s\in \{-1, 1\}}\big\|(t-t_1)^{\g-1}D_N^{-1}Y_{\s, N}(t_1)\big\|_{L^p_\o L^{\infty}_t([0,T])L^1_{t_1}([0,t])H_x^{-\eps}}\notag \\
\intertext{By H\"older's inequality in $t_1$,  we continue with}
& \les
\sum_{\s\in \{-1, 1\}}
\| D_N^{-1}Y_{\s, N} \|_{L^p_\o L^p_TH_x^{-\eps}}, 
\label{fa3}
\end{align}

\noi
provided that $p > \frac{1}{\g}$.

By applying Fubini's theorem and H\"older's inequality, it suffices to estimate
\[\bigg\|\int_0^{t_1} e^{i \s (t_1- t_2) D_N}   (t_1-t_2)^{-\g} D_N^{-1} \P_N d W (t_2)\bigg\|_{L^p_\o H_x^{-\eps}},\]
uniformly in $t_1\in[0,T]$.
From Minkowski's integral inequality (for $p \geq 2$), 
the Wiener chaos estimate (Lemma \ref{LEM:hyp}), 
and \eqref{eq:Dn}, we estimate this term by 
\begin{align}
\begin{split}
\bigg\|\int_0^{t_1} & 
e^{i \s (t_1- t_2) D_N} 
(t_1-t_2)^{-\g}  D_N^{-1} \P_N d W (t_2)\bigg\|_{H^{-\eps}_xL^2_\o}\\
& \sim \bigg[\sum_{|n|\leq N}\jb{n}^{-2\eps}\big(D_N(n)\big)^{-2}\int_0^{t_1} (t_1-t_2)^{-2\gamma}dt_2\bigg]^{\frac 12}\\
& \les \ld_N^{-\frac{\eps}4}
\end{split}
\label{fa4}
\end{align}

\noi
uniformly for all sufficiently large $N \gg 1$, 
provided that $\g < \frac 12$.
The desired bound  \eqref{stcon} follows from \eqref{fa3} and \eqref{fa4}.
This completes the proof of Part (ii).
\end{proof}

\subsection{On the deterministic linear solution}
\label{SUBSEC:lin2}

In \cite{OPT1}, 
the authors exploited a rapid oscillation to show 
that a deterministic linear solution tends to 0 as a space-time distribution.
In the following lemma, 
by using  a rapid oscillation to evaluate relevant oscillatory integrals,
we show that the deterministic linear solution
$v_N^{\lin}$ defined in  \eqref{lin1} converges 
to 0 in $H^{-\eps}([0,T]; H^{1-\eps}(\T^2))$, $\eps > 0$.
This is the last ingredient for  the proof of Theorem~\ref{THM:strong}.

\begin{lemma}\label{LEM:LWN}
Given $(v_0,v_1)\in\H^1(\T^2)$, 
let  $v_N^{\lin}$ be the solution to the linear wave equation 
with $(v, \dt v)|_{t = 0} = (v_0,  v_1)$ defined in  \eqref{lin1}. 
Then, given  any $T>0$ and $\eps > 0$, $v_N^{\lin}$ converges to $0$ in
$H^{-\eps}([0,T]; H^{1-\eps}(\T^2))$
as $N \to \infty$.

\end{lemma}
\begin{proof}
Fix  $\chi\in C^{\infty}_c(\R)$ such that $\chi\equiv 1$ on $[0,1]$
and set  $\chi_{_T}(t) = \chi(T^{-1} t)$ for $T > 0$. 
By setting 
\[V_0 = e^{-\frac t2}\cos\big(tD_N)v_0 
\qquad  \text{and} \qquad V_1 = \D_N(t)(v_1+\tfrac12 v_0),\]

\noi
we have 
$v_N^{\lin} = V_0 +V_1$.
Then, from the definition \eqref{def1}
and H\"older's inequality in time, we have
\begin{align}
\begin{split}
\|v_N^{\lin}\|_{H^{-\eps}_TH^{1-\eps}_x}
& \leq \|\chi_{_T}v_N^{\lin}\|_{H^{-\eps}_tH^{1-\eps}_x} \\
& \le \|\chi_{_T}V_0\|_{H^{-\eps}_tH^{1-\eps}_x} +C_T \|V_1\|_{L^{\infty}_tH^{1-\eps}_x} .
\end{split}
\label{lin5}
\end{align}

In view of \eqref{D2} with \eqref{eq:Dn}, 
the second term on the right-hand side of \eqref{lin5} can be estimated  by
\begin{align}
\|\jb{\nb}^{1-\eps}D_N^{-1}(v_0+v_1)\|_{L^2_x}\les \lambda_N^{-\frac \eps2}\|(v_0,v_1)\|_{\H^1}
\too 0
\label{lin6}
\end{align}

\noi
as $N \to \infty$.
As for the first term, we have 
\begin{align*}
\F_{t,x}(\chi_{_T}V_0)(\tau,n) 
&= \int_{\R}\chi_{_T}(t)e^{-\frac t2}\cos\big(tD_N(n))e^{-it\tau}\ft{v}_0(n)dt\\
&=\frac12\Big[\F_t(\chi_{_T}e^{-\frac t2})(\tau-D_N(n))+\F_t(\chi_{_T}e^{-\frac t2})(\tau+D_N(n))\Big]\ft{v}_0(n).
\end{align*}

\noi
Here, 
$\F_t$ and  $\F_{t,x}$
denote the temporal and space-time  Fourier transforms, respectively.
Integrating by parts and using the properties of $\chi_{_T}$, 
 we have
\begin{align*}
\big|\F_t(\chi_{_T}e^{-\frac t2})(\tau)\big| &= \bigg|\int_{\R}\chi_{_T}(t)e^{-\frac t2}e^{-it\tau}dt\bigg|=\jb{\tau}^{-2M}\bigg|\int_{\R}(1-\dt^2)^M\big[\chi_{_T}(t)e^{-\frac t2}\big]e^{-it\tau}dt\bigg|\\
&\le C_{T, M} \jb{\tau}^{-2M}
\end{align*}

\noi
for  $M \in \Z_{\ge 0}$ (and hence for any $M \geq0$), 
uniformly in  $\tau \in \R$. 
Therefore, we obtain
\begin{align}
\|\chi_{_T}V_0\|_{H^{-\eps}_tH^{1-\eps}_x}^2
&\les \sum_{|n|\leq N}\jb{n}^{2(1-\eps)}|\ft{v}_0(n)|^2
\sum_{\s \in \{-1, 1\}}
\int_{\R}\jb{\tau}^{-2\eps}\big|\F_t(\chi_{_T}e^{-\frac t2})(\tau + \s D_N(n))\big|^2d\tau \notag \\
&\les\sum_{|n|\leq N}\jb{n}^{2(1-\eps)}|\ft{v}_0(n)|^2
\sum_{\s \in \{-1, 1\}}
\int_{\R}\jb{\tau}^{-2\eps}\jb{\tau+\s D_N(n)}^{-1}d\tau \notag \\
&\les\sum_{|n|\leq N}\jb{n}^{2(1-\eps)}|\ft{v}_0(n)|^2\jb{D_N(n)}^{-\eps} \notag \\
&\les \lambda_N^{-\frac \eps2}\|v_0\|_{H^1}^2
\too 0
\label{lin7}
\end{align}

\noi
as $N \to \infty$,
where in the penultimate  step we used the estimate
\[\int_{\R}\jb{\tau}^{-a}\jb{\tau-\tau_0}^{-b}d\tau \les \jb{\tau_0}^{1-a-b}\]

\noi
 for any $\tau_0\in\R$ and any $a,b<1$ with $a+b>1$; 
see for example \cite[Lemma~4.2]{GTV}.
Putting  \eqref{lin5}, \eqref{lin6}, and \eqref{lin7} together, 
we conclude that 
 $v_N^{\lin}$ converges to 0 in
$H^{-\eps}([0,T]; H^{1-\eps}(\T^2))$
as $N \to \infty$.
\end{proof}

\section{Trivial limit in the strong noise case}
\label{SEC:tri}

In this section, we prove triviality in the strong noise case
(Theorem~\ref{THM:strong}).
In particular, we assume \eqref{C1} in the following.
As described in Section \ref{SEC:1}, 
we apply the Da Prato-Debussche trick
and work in terms of the residual term 
 $v_N = u_N-z_N$.
From 
\eqref{NLW4}, \eqref{non1}, 
\eqref{Herm}, and \eqref{Wick1}, 
we see that 
$v_N$ satisfies
\begin{equation}\label{SdNLW2}
\begin{cases}
\L_N  v_N+v_N^3 + 3   v_N^2z_N + 3  v_N :\!  z_N^2 \!:+ :\!  z_N^3\!:=0 \\
(v_N, \dt v_N)|_{t = 0} = (v_0,v_1).
\end{cases}
\end{equation}

\noi
The main idea is to use  the decay properties of the Wick powers
$:\!z_N^\l\!:$ and the deterministic linear solution $v_N^{\lin}$
proved in Section \ref{SEC:2}.

We first establish almost sure global well-posedness
of \eqref{SdNLW2}.
Given $s \in \R$ and $T > 0$, 
define the solution space $X^s(T)$
by setting
\begin{equation}\label{X}
X^s(T)  \stackrel{\textup{def}}{=} C([0,T];H^s(\T^2))\cap C^1([0,T];H^{s-1}(\T^2)).
\end{equation}

\begin{proposition}\label{PROP:GWP}
Let $N \in \N$. The Cauchy problem \eqref{SdNLW2} is almost surely globally well-posed in $\H^1(\T^2)$. More precisely, given  any $(v_0,v_1)\in\H^1(\T^2)$  and any $T>0$, 
there exists a set $\O_{T}\subset \O$ of full probability such that,
 for any $\o\in\O_{T}$
 and $N \in \N$, there exists a unique solution $v_N\in X^1(T)$  to \eqref{SdNLW2}.
\end{proposition}

We recall the following lemma from \cite{GKO}.
\begin{lemma} \label{LEM:biest}
Let $0 \le s \le 1$.

\noi
\textup{(i)}
 Suppose that $1<p_j,q_j,r < \infty$, $\frac{1}{p_j}+\frac{1}{q_j}=\frac{1}{r}$, $j=1,2$.
Then, we have
\[
\| \jb{\nb}^{s} (fg) \|_{L^r (\T^d)}
\les \| f \|_{L^{p_1}(\T^d)} \| \jb{\nb}^{s} g \|_{L^{q_1} (\T^d)} +\| \jb{\nb}^s f \|_{L^{p_2} (\T^d)} \| g \|_{L^{q_2} (\T^d)}.
\]

\smallskip
\noi
\textup{(ii)} 
Suppose that $1<p,q,r < \infty$ satisfy the scaling condition $\frac{1}{p}+\frac{1}{q}\leq\frac{1}{r} + \frac{s}{d}$.
Then, we have
\[
\| \jb{\nb}^{-s} (fg) \|_{L^r (\T^d)}
\les \| \jb{\nb}^{-s} f \|_{L^p (\T^d)} \| \jb{\nb}^s g \|_{L^q (\T^d)}.
\]
\end{lemma}

The first estimate is a consequence of the Coifman-Meyer theorem
and the transference principle.
See \cite{GKO} for the references therein.
Note that
while  the second estimate
was shown only for 
$\frac{1}{p}+\frac{1}{q}= \frac{1}{r} + \frac{s}{d}$
in \cite{GKO}, 
the general case
$\frac{1}{p}+\frac{1}{q}\leq\frac{1}{r} + \frac{s}{d}$
follows from a straightforward modification.

We now present a proof of Proposition \ref{PROP:GWP}.

\begin{proof}[Proof of Proposition \ref{PROP:GWP}]
Let  $(v_0,v_1)\in\H^1(\T^2)$
and fix a target time $T > 0$ as in the statement. 
We first briefly go over local well-posedness of \eqref{SdNLW2}
with a  control on $[0, T]$.
By writing \eqref{SdNLW2}  in the Duhamel formulation, we have 
\begin{align}
\begin{split}
v_N(t) & = \G_N (v_N) (t)\\ &\stackrel{\textup{def}}{=} \dt \D_N(t) v_0 + \D_N (t) (v_0+v_1) \\
&\quad - \int_0^t \D_N(t-t') \big( v_N^3 + 3   v_N^2 z_N + 3  v_N :\!  z_N^2 \!: + :\!  z_N^3\!: \big) (t') dt'.
\end{split}
\label{D4a}
\end{align}

Let $D_N(n)$ be as in \eqref{Dn}.
Recall from Lemma \ref{LEM:CN} that we have $\ld_N > 0$.
Then, by separately estimating the cases $D_N \ges 1$ and $D_N \ll 1$, 
we have 
\begin{align}
\bigg|e^{-\frac{t}{2}} \frac{\sin t D_N(n)}{D_N(n)} \bigg|
\les \jb{n}^{-1}
\label{D4}
\end{align}

\noi
for any $N \geq 1$, $n \in \Z^2$, and $t \geq 0$.
Hence, in view of \eqref{D2}, we have 
\begin{align}
\|\G_N(v_N)\|_{X^1(\dl)} \les \|(v_0,v_1)\|_{\H^1} 
+ \|v_N^3 + 3   v_N^2   z_N + 3 v_N :\!  z_N^2 \!: + :\!  z_N^3\!:\|_{L^1_\dl L^2_x}
\label{D5}
\end{align}

\noi
for any $\dl > 0$.

Next, observe that from its definition \eqref{eq:zN}, $z_N$ satisfies
\[z_N = \P_N z_N,\]
which implies that  we  have
\begin{align*}
\,:\!z_N^\l\!:\, = \P_{\l N}:\!z_N^\l\!:\,
\end{align*}

\noi
for $\l= 2,3$.
Hence, by H\"older's, Sobolev's and Bernstein's inequalities with the frequency support
property of the Wick powers, we obtain
\begin{align} 
\begin{split}
\|v_N^3 & + 3    v_N^2 z_N + 3  v_N :\!  z_N^2 \!:+ :\!  z_N^3\!:\|_{L^1_\dl L^2_x}\\ 
&\les \dl^{\tfrac12}\Big(\|v_N\|_{L^{\infty}_\dl L^6_x}^3+\|v_N\|_{L^{\infty}_\dl L^4_x}^2\|z_N\|_{L^{2}_\dl L^{\infty}_x}\\
& \hphantom{X}
+\|v_N\|_{L^{\infty}_\dl L^2_x}\|:\!z_N^2\!:\|_{L^{2}_\dl L^{\infty}_x}+\|:\!z_N^3\!:\|_{L^{2}_\dl L^{\infty}_x}
\Big) \\ 
&\les \dl ^{\tfrac12}\Big(\|v_N\|_{X^1(\dl )}^3+N^{\eps}\|v_N\|_{X^1(\dl )}^2\|z_N\|_{L^{2}_\dl W^{-\eps,\infty}_x}\\
&\hphantom{X}
+N^{\eps}\|v_N\|_{X^1(\dl )}\|:\!z_N^2\!:\|_{L^{2}_\dl W^{-\eps,\infty}_x}+N^{\eps}\|:\!z_N^3\!:\|_{L^{2}_\dl W^{-\eps,\infty}_x}\Big)
\end{split}
\label{D6}
\end{align}

\noi
for $0 < \dl \leq 1$.

Given a large target time $T > 0$, 
$M \geq 1$, and $N \in \N$, 
we set 
\begin{align*}
\O_{N,T}^M 
= \Big\{\o \in \O: \|:\!z_N^\l\!:\|_{L^{2}_TW^{-\eps,\infty}_x}\leq M,~\l=1,2,3\Big\}.
\end{align*}

\noi
Then, for any $\o \in \O_{N,T}^M$, 
it follows from \eqref{D5} and \eqref{D6} that 
\begin{align*}
\|\G_N(v_N)\|_{X^1(\dl)}
&\leq C_0\|(v_0,v_1)\|_{\H^1} + C_1\dl^{\tfrac12}\Big(\|v_N\|_{X^1(\dl)}^3\\&\quad +N^{\eps}M\|v_N\|_{X^1(\dl)}^2+N^{\eps}M\|v_N\|_{X^1(\dl)}+N^{\eps}M\Big).
\end{align*}
In particular, if we set
\[R = 1+2C_0\|(v_0,v_1)\|_{\H^1}
\qquad \text{and}
\qquad 
\dl_{N, R} = (100C_1R^2N^{\eps}M)^{-2},\]

\noi
then
we see that $\G_N$ maps the ball 
$B_{N,R} = \{v_N:~\|v_N\|_{X^1(\dl_{N,R})}\leq R\}$ into itself.
Furthermore, by a similar computation, 
we can show that $\G_N$ is 
 a contraction on $B_{N, R}$,
 establishing existence of a unique solution $v_N \in B_{N,R}$ to \eqref{SdNLW2}.
 A standard continuity argument allows us to extend  the uniqueness to the whole space $X^1(\dl_{N,R})$.

It follows from 
 \eqref{tail-z} and  Chebyshev's inequality 
 (as in \cite[Lemma 3]{BOP1}\footnote{Lemma 2.2 in the arXiv version. 
 See also Lemma 4.5 in \cite{Tzvet}.})
that 
\begin{align}P\Big( \|:\!z_N^\l\!:\|_{L^{2}_TW^{-\eps,\infty}_x}>M\Big) 
 \leq Ce^{-cM^{\frac2\l}T^{-\frac1{\l }}\ld_N^{\frac{\eps}{2\l}}}.
\label{D8}
 \end{align}

\noi
Then, defining $\O_{T}$ by 
\begin{align*}
\O_{T} = \bigcap_{N \in \N} \O_{N,T}
= \bigcap_{N \in \N} \bigcup_{M\in\N}\O_{N,T}^M,
\end{align*}

\noi
it follows from \eqref{D8} that  $\O_{N, T}$ has probability 1
and therefore $\O_T$ is a set of full probability.
Furthermore, given $ \o \in \O_T$
and $N \in \N$, there exists $M = M(N) \in \N$
such that $\o \in \O^M_{N, T}$
and thus the argument above shows local existence of a unique solution $v_N$
to \eqref{SdNLW2}
on the time interval $[0, \dl_{N,R}(\o)]$.
 This proves almost sure local well-posedness of~\eqref{SdNLW2}.
Note that  we have the following blowup alternative for the maximal time $T_{N, R}^*
= T_{N, R}^*(\o)$
of existence; 
given $\o \in \O_T$ and $N \in \N$, 
we have either 
\begin{align}
\lim_{t\nearrow T_{N,R}^*}\|v_N\|_{X^1(t)} = \infty 
\qquad \text{or} \qquad T_{N,R}^* \geq T.
\label{D9}
\end{align}

Next, we prove almost sure well-posedness
on the entire time interval $[0, T]$.
We follow the argument introduced by Burq and Tzvetkov \cite{BT-JEMS}
in the context of random data global well-posedness
of the cubic NLW on $\T^3$.
In view of the blowup alternative \eqref{D9}, 
it suffices to show that, for each $\o \in \O_T$, 
 the $\H^1$-norm of $(v_N(t), \dt v_N(t))$
remains finite on $[0, T]$.

Define the energy $\EE_N(v)$ by setting 
\[\EE_N(v)(t)=\frac{1}{2}\|\nb v(t)\|_{L^2}^2 + \frac12\|\dt v(t)\|_{L^2}^2 + \frac{1}{4}\|v(t)\|_{L^4}^4
+ \frac{1}{2} \ld_N \|v(t)\|_{L^2}^2.\]

\noi
Then, for a solution $v_N$ to \eqref{SdNLW2}, 
we have 
\begin{align*}
\dt\EE(v_N)
& = -\int_{\T^2}\dt v_N \Big(\dt v_N 
+ 3   v_N^2 z_N + 3  v_N :\!  z_N^2 \!: + :\!  z_N^3\!: \Big)dx\\
& \les 
- \|\dt v_N\|_{L^2}^2  + 
N^\eps \|\dt v_N\|_{L^2}
\Big(  \|v_N\|_{L^4}^2\|z_N\|_{W^{-\eps,\infty}}\\
& \hphantom{X}+  \|v_N\|_{L^4}\|:\!z_N^2\!:\|_{W^{-\eps, \infty}}
 + \|:\!z_N^3\!:\|_{W^{-\eps,\infty}}\Big)\\
 \intertext{By Young's inequality,}
&  \les \Big(1 +  N^{\eps} \|z_N\|_{W^{-\eps,\infty}}\Big) \EE(v_N)
+ N^{4\eps} \|:\!z_N^2\!:\|_{W^{-\eps,\infty}}^4
+ N^{2\eps}\|:\!z_N^3\!:\|_{W^{-\eps,\infty}_x}^2.
\end{align*}

\noi
Then, it follows from 
Gronwall's inequality that
given $T> 0$ and $N, M \in \N$,  there exists a constant $C(N,T,M)>0$ such that for any $\o\in\O_{N,T}^M$,
we have
\[\|v_N\|_{X^1(T_{N,R}^*)}\les
\sup_{t\in [0,T_{N,R}^*)}
\EE(v_N)(t)\leq C(N,T,M)\EE(v_N)(0) < \infty.\]

\noi
Since the choices of $N$ and $M$ are arbitrary,
this implies that  $T_{N,R}^*(\o) \geq T$ for any $\o \in \O_T$.
This completes the proof of Proposition \ref{PROP:GWP}.
\end{proof}

We are now ready to present a proof of Theorem \ref{THM:strong}.

\begin{proof}[Proof of Theorem~\ref{THM:strong}]
Let  $(v_0,v_1)\in\H^1(\T^2)$.
Fix $T>0$.
Given $N \in \N$, set 
\begin{align}
V_N = v_N - v_N^{\lin},
\label{E0}
\end{align}

\noi
where  $v_N^{\textup{lin}}$ is the linear solution 
defined in \eqref{lin1}.
Proposition~\ref{PROP:GWP} ensures 
that $V_N$ exists almost surely on the time interval $[0, T]$, 
where it  satisfies the Duhamel formulation.
In the following, we show that $V_N$
tends to 0
in $C([0,T];H^{1-\eps}(\T^2))$.

Fix $\eps > 0$ sufficiently small.
Then, from  Lemma \ref{LEM:biest}, we have
\begin{align}
\begin{split}
\|   z_N  (V_N+v_N^{\textup{lin}})^2 \|_{L_T^1 H_x^{-\eps}}
&\les T^{\frac{1}{2}} \| z_N \|_{L_T^2 W_x^{-\eps,\frac{1}{\eps}}} 
\big\| \jb{\nb}^{\eps} \big[(V_N+v_N^{\textup{lin}})^2\big] \big\|_{L_T^{\infty} L_x^{\frac{2}{1-\eps}}} \\
&\les T^{\frac{1}{2}} \| z_N \|_{L_T^2 W_x^{-\eps,\infty}} \| \jb{\nb}^{\eps} (V_N+v_N^{\textup{lin}}) \|_{L_T^{\infty} L_x^{\frac{4}{1-\eps}}}^2\\&
\les T^{\frac{1}{2}} \| z_N \|_{L_T^2 W_x^{-\eps,\infty}} \| V_N+v_N^{\textup{lin}} \|_{C_TH_x^{1-\eps}}^2.
\end{split}
\label{E1}
\end{align}

\noi
Similarly, we have
\begin{align}
\begin{split}
\| :\!  z_N^2 \!: (V_N+v_N^{\textup{lin}}) \|_{L_T^1 H_x^{-\eps}}
&\les T^{\frac{1}{2}} \| :\!  z_N^2 \!: \|_{L_T^2 W_x^{-\eps,\frac{1}{\eps}}} \| \jb{\nb}^{\eps} (V_N+v_N^{\textup{lin}})\|_{L_T^{\infty} L_x^{\frac{2}{1-\eps}}} \\
&\les T^{\frac{1}{2}} \| :\!  z_N^2 \!: \|_{L_T^2 W_x^{-\eps,\infty}} \| V_N+v_N^{\textup{lin}} \|_{C_TH_x^{1-\eps}}. 
\end{split}
\label{E2}
\end{align}

\noi
By  H\"older's inequality, we have
\begin{equation}
\| :\!  z_N^3 \!: \|_{L_T^1  H_x^{-\eps}}
\le T^{\frac{1}{2}} \| :\!  z_N^3 \!: \|_{L_T^2  H_x^{-\eps}}. 
\label{E3}
\end{equation}

In order to  estimate the term $(V_N+v_N^{\textup{lin}})^3$, 
we  use \eqref{eq:Dn} by assuming that $N \gg 1$
such that  $\D_N$ is bounded from $L^2(\T^2)$ to $C([0,T];H^{1-\eps}(\T^2))$
with norm less than $\ld_N^{-\frac \eps2}$.\footnote{Note that  this gain of $\lambda_N^{-\frac \eps2}$ is not true for $\dt\D_N$.
This is the reason we only prove convergence of $V_N$ in $C([0,T];H^{1-\eps}(\T^2))$ instead of the smaller space $X^{1-\eps}(T)$.} 
Then,  H\"older's and Sobolev's inequalities yield
\begin{equation}
\begin{split} 
\bigg\|\int_0^t\D_N(t-t')\big[(V_N+v_N^{\textup{lin}})^3(t')\big]dt'\bigg\|_{C_T H_x^{1-\eps}}&\les \ld_N^{-\frac \eps2}\| (V_N+v_N^{\textup{lin}})^3 \|_{L_T^1 L^2_x}
\\&\le \ld_N^{-\frac \eps2} T \| V_N+v_N^{\textup{lin}} \|_{L_T^{\infty} L_x^6}^3
\\&\les \ld_N^{-\frac \eps2} T \| V_N+v_N^{\textup{lin}} \|_{C_T H_x^{1-\eps}}^3.
\end{split}
\label{E4}
\end{equation}

\noi
Moreover, from \eqref{D4}, we have 
\begin{align}
\|v_N^{\textup{lin}}\|_{C_T H_x^{1-\eps}}\les \|(v_0,v_1)\|_{\H^{1-\eps}},
\label{E5}
\end{align}
uniformly in $N$.

From \eqref{E0}, we have 
\[V_N=\G_N(V_N+v_N^{\textup{lin}})-v_N^{\textup{lin}}\] 

\noi
where $\G_N$ is as in \eqref{D4a}.
Then, putting \eqref{E1} - \eqref{E5}  together
along with 
 the boundedness of $\D_N$ from $H^{-\eps}(\T^2)$ to $C([0,T];H^{1-\eps}(\T^2))$, 
we obtain
\begin{align}
\| V_N \|_{C_T H_x^{1-\eps}}
&\les   \ld_N^{-\frac \eps2}T \big(\| V_N \|_{C_T H_x^{1-\eps}}+\|(v_0,v_1)\|_{\H^1}\big)^3 \notag \\
&\quad + T^\frac{1}{2}\Big(\| z_N \|_{L_T^2 W_x^{-\eps,\infty}} 
\big(\| V_N \|_{C_T H_x^{1-\eps}}+\|(v_0,v_1)\|_{\H^1}\big)^2 \label{GN}
\\
&\quad + \| :\!  z_N^2 \!: \|_{L_T^2 W_x^{-\eps,\infty}} \big(\| V_N \|_{C_T H_x^{1-\eps}}+\|(v_0,v_1)\|_{\H^1}\big) + \| :\!  z_N^3 \!: \|_{L_T^2  H_x^{-\eps}} \Big).\notag 
\end{align}

As in  \cite{HRW},  we introduce a sequence of stopping times
\begin{align}
\tau_N^{\rho} = T\wedge \inf\big\{\tau \geq 0:~\|V_N\|_{C_{\tau}H^{1-\eps}_x}>\rho\big\}
\label{E7}
\end{align}

\noi
for $\rho>0$.
Then, 
the bound \eqref{GN} and the continuity in time of $V_N$ (with values in $H^{1-\eps}(\T^2)$) show that for any $\rho>0$,
\[\begin{split}
\| V_N \|_{C_{\tau_N^\rho}H^{1-\eps}_x}
&\les   \ld_N^{-\frac \eps2} T(\rho+\|(v_0,v_1)\|_{\H^1})^3 \\&
\quad + T^\frac{1}{2}\Big(\| z_N \|_{L_T^2 W_x^{-\eps,\infty}} \big(\rho+\|(v_0,v_1)\|_{\H^1}\big)^2 \\
&\quad + \| :\!  z_N^2 \!: \|_{L_T^2 W_x^{-\eps,\infty}} \big(\rho+\|(v_0,v_1)\|_{\H^1}\big) + \| :\!  z_N^3 \!: \|_{L_T^2  H_x^{-\eps}} \Big).
\end{split}\]

\noi
Taking an expectation, we conclude from Lemma~\ref{LEM:CN} and Proposition~\ref{PROP:estzN} that
\[\lim_{N\rightarrow\infty}\E\big[\|V_N\|_{C_{\tau_N^\rho}H^{1-\eps}_x}\big]=0.\]

\noi
When $\tau_N^\rho < T$, 
it follows from the definition \eqref{E7} of  $\tau_N^{\rho}$
and 
the continuity in time of $V_N$ 
that 
\[ \|V_N\|_{C_{\tau_N^\rho}H^{1-\eps}_x} = \rho.\]

\noi
Hence, we obtain
\[P(\tau_N^\rho<T) \leq  \frac{1}{\rho}
\E\Big[\|V_N\|_{C_{\tau_N^\rho}H^{1-\eps}_x}\ind_{[0,T)}(\tau_N^\rho)\Big]\leq \frac{1}{\rho}\E\big[\|V_N\|_{C_{\tau_N^\rho}H^{1-\eps}_x}\big]
\too 0\] 

\noi
as $N \to \infty$.
This in turn implies that, for any $\rho > 0$,  
we have 
\begin{align}
P(\|V_N\|_{C_T H_x^{1-\eps}}>\rho)= P(\tau_N^\rho< T)
\too 0 
\label{E8}
\end{align}
as $N \to \infty$. 

Finally, recalling the decompositions $u_N = z_N +v_N$
and \eqref{E0}
and applying  the embedding $C([0,T];H^{s}(\T^2))\subset H^{-\eps}([0,T];H^{s}(\T^2))$ for any $s\in\R$, 
we obtain
\begin{align*}
\|u_N\|_{H^{-\eps}_TH^{-\eps}_x} 
& = \|z_N+v_N^{\lin}+V_N\|_{H^{-\eps}_TH^{-\eps}_x}\\
& \les \|z_N\|_{C_TH^{-\eps}_x}+\|v_N^{\lin}\|_{H^{-\eps}_TH^{1-\eps}_x}+\|V_N\|_{C_TH^{1-\eps}_x}.
\end{align*}

\noi
The first and third terms on the right-hand side converge to 0 in probability by 
Proposition~\ref{PROP:estzN}\,(ii)
and \eqref{E8}, respectively, 
 while the second term  on the right-hand side 
converges to 0   by Lemma~\ref{LEM:LWN}.
 This completes the proof of Theorem~\ref{THM:strong}.
\end{proof}

\section{Deterministic limit in the weak noise case}

In this section, we work in the weak noise case:
\begin{align}
\lim_{N \to \infty} \al_N^2 \log N = \kk^2 \in [0, \infty)
\label{f0}
\end{align}

\noi
 and present a proof of Theorem~\ref{THM:weak}.
First, note that by setting $\ld_N = 1$, 
the results in Section \ref{SEC:2} hold in this case.
In particular,  the linear stochastic  wave equation 
\eqref{dSLWN2} admits a unique invariant measure, still denoted by $\mu_N$.

Let $(z_{0,N}^{\o},z_{1,N}^{\o})$ be as in \eqref{IV3}, distributed by the Gaussian measure $\mu_N$.
Denote by  $z_N$ the solution to \eqref{dSLWN2} with  $(z_N, \dt z_N)|_{t = 0} = (z_{0,N}^{\o},z_{1,N}^{\o})$.
Then, by invariance of $\mu_N$, the variance of $z_N(t)$ is given by
\begin{align}
\s_N \stackrel{\text{def}}{=} \frac{\al_N^2}{8\pi^2}\sum_{|n| \le N} \frac{1}{\jb{n}^2}.
\label{f1}
\end{align}

\noi
We now define 
the Wick powers $:\!z_N^\l\!:$ as in \eqref{Wick1}
with this new variance $\s_N$ defined in \eqref{f1}.
Note that  from  \eqref{f1} with \eqref{f0} and Lemma \ref{LEM:log}, we have
\begin{equation}\label{eq:dlN}
\lim_{N \to \infty} \s_N = \frac{1}{4\pi} \kk^2.
\end{equation}

As in the proof of Theorem \ref{THM:strong}, 
we proceed with the Da Prato-Debussche trick.
Namely, write the solution $u_N$ to \eqref{NLW2} as $u_N = v_N+z_N$.
Then, the residual term $v_N$ satisfies
\begin{equation}\label{SdNLW3}
\begin{cases}
\L v_N+ (3\s_N-1) (v_N+z_N) + v_N^3 + 3  v_N^2  z_N + 3  v_N :\!  z_N^2 \!: + :\!  z_N^3\!:\, =0 \\
(v_N, \dt v_N)|_{t = 0} = (v_0,v_1), 
\end{cases}
\end{equation}

\noi
where 
 $\L = \dt^2-\Dl +\dt +1$ is as in \eqref{C2a}.

Proceeding as in the proof of 
Proposition  \ref{PROP:estzN}, 
we obtain  the following lemma
on the regularity and decay properties of the Wick powers
$:\!z_N^\l\!:$.

\begin{lemma}\label{LEM:estzN3}
Let $\l \in \N$. Given any finite  $p,q \geq 1$,  $T>0$, and $\eps>0$,
we have\footnote
{In this case,  we also have convergence of $\dt z_N$ to 0 in $C([0,T];H^{-1-\eps}(\T^2))$ 
since the convergence to 0 comes from $\alpha_N\to  0$, 
 not from a gain of a negative power of $\ld_N$.}
 \[
\lim_{N \to \infty} \E \Big[ \| :\! z_N^\l (t)\!: \|_{L^q_T W_x^{-\eps,\infty}}^p \Big] =0
 \qquad
 \text{and}
 \qquad 
\lim_{N \to \infty} \E \Big[\| z_N (t) \|_{X^{-\eps}(T)}^p \Big] = 0, 
\]

\noi
 where $X^s(T)$ is as in \eqref{X}.
\end{lemma}

Lemma \ref{LEM:estzN3} follows as in Proposition  \ref{PROP:estzN}
once we note the following;
under \eqref{f0}, we have  $\al_N \to 0$ as $N \to \infty$, 
which yields
\begin{align*}
\E \big[ |\jb{\nb}^{-\eps} :\! z_N^\l (t, x))\!:|^2 \big]
&= \l ! \sum_{n_1, \dots, n_{\l} \in \Z^2_N} \bigg( \prod_{j=1}^{\l} \frac{\al_N^2}{\jb{n}^2} \bigg) \jb{n_1+\dots+n_{\l}}^{-2\eps}
\les_{\l} \al_N^{2\l}\\
& \too 0,
\end{align*}

\noi
as $N \to \infty$.

By arguing  as in the proof of Proposition~\ref{PROP:GWP}, 
we can  show that the equation \eqref{SdNLW3} is almost surely globally well-posed in $\H^1(\T^2)$ in the sense that for any $T>0$,  there exists a set $\O_T$ of full probability  such that for any $\o\in \O_{T}$ 
and $N \in \N$, there exists a unique solution $v_N \in X^1(T)$ to \eqref{SdNLW3},  
satisfying the bound
\begin{equation*} \label{est:vn}
\| v_N \|_{X^1(T)} \le C(N,T,\o) \| (v_0,v_1) \|_{\H^1}.
\end{equation*}

Our main goal in this section is to prove the following proposition.

\begin{proposition}\label{PROP:decay}

Let $v_N$ be the solution to \eqref{SdNLW3}.
Then, given any $T,\eps>0$, 
$v_N$ converges in probability to 
the  solution $w_{\kk}$ to \eqref{dNLWw}
in $X^{1-\eps}(T)$.
\end{proposition}

Once we have Proposition \ref{PROP:decay}, 
Theorem \ref{THM:weak} 
follows from the decomposition $u_N = z_N + v_N$
and the decay of $z_N$ to 0 in $X^{-\eps}(T)$ presented in Lemma \ref{LEM:estzN3}.
Hence, it remains to prove 
Proposition \ref{PROP:decay}.

\begin{proof}[Proof of Proposition \ref{PROP:decay}]
Fix $T> 0$. By proceeding as in the proof of  Proposition~\ref{PROP:GWP}, 
we can show that 
the deterministic equation \eqref{dNLWw} admits a unique global solution $w_{\kk}\in X^1(T)$, 
satisfying the energy bound 
\begin{equation} \label{est:w}
\| w_{\kk} \|_{X^1(T)} \le R_\kk \stackrel{\text{def}}{=} C_{\kk}(T) \| (v_0,v_1) \|_{\H^1}.
\end{equation}

Define $\be_N$ by setting
\begin{align}
\be_N = 3 \bigg(\s_N - \frac{\kk^2}{4\pi}\bigg).
\label{f0a} 
\end{align}

\noi
Then, 
we rewrite \eqref{SdNLW3} as
\begin{align*}
&\dt^2 v_N - \Dl v_N +\dt v_N + \frac3{4\pi} \kk^2 v_N + v_N^3 
+ \QQ_N(v_N) =0, 
\end{align*}

\noi
where $\QQ_N(v_N)$ is the ``error'' part given by 
\begin{align*}
\QQ_N(v_N) =  \be_N v_N + (3\s_N-1) z_N + 3 v_N^2 z_N + 3  v_N :\!  z_N^2 \!: + :\!  z_N^3 \!:.
\end{align*}

\noi
By setting $V_N = v_N - w_{\kk}$, we see that $V_N$ then solves
\begin{equation}
\begin{split}
\begin{cases}
\dt^2 V_N  - \Dl V_N +\dt V_N + \frac3{4\pi} \kk^2 V_N 
+ V_N^3 \\
\hphantom{XX}+3 V_N^2 w_{\kk} + 3 V_N w_{\kk}^2 
 + \QQ_N(V_N + w_\kk) = 0\\
(V_N,\dt V_N)\big|_{t=0}=(0,0).
\end{cases}
 \end{split}\label{eq:wN}
 \end{equation}

\noi
We first establish a good control on $V_N$
on short time intervals.
With a slight abuse of notations, 
we set 
\begin{equation*}
X^s(I)  \stackrel{\textup{def}}{=} C(I;H^s(\T^2))\cap C^1(I;H^{s-1}(\T^2))
\end{equation*}

\noi
for an interval $I \subset \R_+$.

 \begin{lemma}\label{LEM:wN}
Given $\kk$ as in \eqref{f0}, let $R_\kk$ be as in  \eqref{est:w}. 
Then, for any $\rho >0$ and small $\eps > 0$, there exist $T_0=T_0(\rho,R_\kk)$ and $C_0>0$ such that  if 
\begin{align}
\|V_N\|_{X^{1-\eps}([t_0, t_0 + \tau])}\leq\rho
\label{f3}
\end{align}

\noi
for some $t_0 \in [0, T)$ and $0 < \tau \leq T_0$
such that $t_0 + \tau \leq T$, 
then we have
\begin{equation}
\begin{split}
\| V_N \|_{X^{1-\eps}([t_0, t_0 + \tau])}
&\leq C_0\Big\{\|(V_N(t_0),\dt V_N(t_0))\|_{\H^{1-\eps}}+ \be_N\tau (\rho+R_\kk) \\
& \quad 
+ \tau\| z_N \|_{L_T^{\infty} H_x^{-\eps}}
+ \tau^{\tfrac12}\big(\| z_N \|_{L_T^2 W_x^{-\eps,\infty}}(\rho^2+R_\kk^2) \\ 
& \quad + \| :\!  z_N^2 \!: \|_{L_T^2 W_x^{-\eps,\infty}}(\rho+R_\kk) + \| :\!  z_N^3 \!: \|_{L_T^2 H_x^{-\eps}}
\big)\Big\}.
\end{split}\label{est:wNT0}
\end{equation}
 \end{lemma}

 \begin{proof}
 Given $t_0 \in [0, T)$ and $ 0 < \tau  \leq T - t_0$, 
 set $I = [t_0, t_0 + \tau]$.
 By estimating the Duhamel formulation of \eqref{eq:wN} on $I$
 as in the previous section, we  have 
\begin{equation*}
\begin{split}
\| V_N \|_{X^{1-\eps}(I)}
& \les \|(V_N(t_0),\dt V_N(t_0))\|_{\H^{1-\eps}}\\
&\quad+ \tau\| V_N \|_{X^{1-\eps}(I)} \big(\| V_N \|_{X^{1-\eps}(I)}^2 
+ \| w_{\kk} \|_{C_TH^1_x}^2\big) \\
& \quad + \be_N \tau \big(\| V_N \|_{X^{1-\eps}(I)} + \| w_{\kk} \|_{C_T H^{-\eps}_x}\big) 
+ (3\s_N-1) \tau  \| z_N \|_{L_T^{\infty} H_x^{-\eps}} \\
& \quad + \tau^{\frac{1}{2}} \Big( \| z_N \|_{L_T^2 W_x^{-\eps,\infty}} 
(\| V_N \|_{X^{1-\eps}(I)}^2 + \| w_{\kk} \|_{C_TH^1_x}^2) \\
& \quad + \| :\!  z_N^2 \!: \|_{L_T^2 W_x^{-\eps,\infty}} (\| V_N \|_{X^{1-\eps}(I)} 
+ \| w_{\kk} \|_{C_TH^1_x}) + \| :\!  z_N^3 \!: \|_{L_T^2 H_x^{-\eps}} \Big),
\end{split}
\end{equation*}
where the first term comes from the contribution of the linear evolution 
associated with the operator
$\mathcal L^\kk = 
\dt^2   - \Dl  +\dt  + \frac3{4\pi} \kk^2$, 
starting from initial data $(V_N(t_0),\dt V_N(t_0))$.
Hence, from \eqref{est:w}
and \eqref{f3}, 
 we obtain
\begin{equation*}
\begin{split}
\| V_N \|_{X^{1-\eps}(I)}
& \les \|(V_N(t_0),\dt V_N(t_0))\|_{\H^{1-\eps}}
+\tau(\rho^2+R_\kk^2)\|V_N\|_{X^{1-\eps}(I)}\\
& \quad +\be_N\tau(\rho+R_\kk) + \tau\| z_N \|_{L_T^{\infty} H_x^{-\eps}}
 + \tau^{\tfrac12}\big(\| z_N \|_{L_T^2 W_x^{-\eps,\infty}}(\rho^2+R_\kk^2) \\ 
& \quad + \| :\!  z_N^2 \!: \|_{L_T^2 W_x^{-\eps,\infty}}(\rho+R_\kk) + \| :\!  z_N^3 \!: \|_{L_T^2 H_x^{-\eps}}
\big), 
\end{split}
\end{equation*}

\noi
where we used the boundedness of  $3\s_N-1$  in view of \eqref{eq:dlN}.\footnote
{The bound on $3\s_N - 1$ depends on the entire sequence $\{\al_N\}_{N \in \N}$
but this does not cause an issue since we work with a {\it fixed} sequence 
$\{\al_N\}_{N \in \N}$.}
Then, by choosing $T_0 = T_0(\rho, R_\kk)> 0$
sufficiently small, we obtain the desired bound~\eqref{est:wNT0}.
\end{proof}

We continue with the proof of Proposition \ref{PROP:decay}.
Fix small $\eps > 0$.
In the following, we proceed as in  the previous section
and   introduce a sequence of stopping times 
\begin{align}
\tau_N^\rho = T\wedge \inf\big\{\tau\geq 0 : \|V_N\|_{X^{1-\eps}(\tau)}>\rho\big\}
\label{f3a}
\end{align}

\noi
for $\rho > 0$.

Let $R_\kk$ and $T_0$ be as in \eqref{est:w} and
  Lemma~\ref{LEM:wN}, respectively.
Given  $j=0,...,\big[\frac{T}{T_0}\big] + 1$,
set  $t_j=jT_0$ for $0 \leq j \leq \big[\frac{T}{T_0}\big]$
and $t_{[\frac{T}{T_0}] + 1} = T$.\footnote
{If $T$ is a multiple of $T_0>0$, 
then we do not need to consider $j = \big[\frac{T}{T_0}\big] + 1$
and it suffices to prove \eqref{f4}
for all $j=0,...,\big[\frac{T}{T_0}\big] - 1$.
}
Then, our goal is to 
apply  Lemma~\ref{LEM:wN} iteratively 
and  show that 
\begin{align}
\lim_{N\rightarrow\infty}P(t_j\leq \tau_N^\rho <t_{j+1})=0,
\label{f4}
\end{align}

\noi
for all $j=0,...,\big[\frac{T}{T_0}\big]$.
Once we prove \eqref{f4}, 
we obtain 
\begin{align*}
P(\|V_N\|_{X^{1-\eps}(T)}>\rho)= P(\tau_N^\rho< T)
\leq \sum_{j = 0}^{\big[\frac{T}{T_0}\big]} 
P(t_j\leq \tau_N^\rho <t_{j+1})
\too 0 
\end{align*}
as $N \to \infty$.

From  the definition \eqref{f3a} of $\tau_N^\rho$, 
the  continuity in time of $(V_N, \dt V_N)$ (with values in $\H^{1-\eps}(\T^2)$), 
and applying 
Lemma~\ref{LEM:wN}
along with  Lemma~\ref{LEM:estzN3}
and $\be_N\rightarrow 0$ (which follows from  \eqref{eq:dlN} and~\eqref{f0a}), we have 
\begin{align}
\begin{split}
P(t_j\leq \tau_N^\rho <t_{j+1})
& =  \frac{1}{\rho}\E\Big[
\|V_N\|_{X^{1-\eps}([t_j, \tau_N^\rho])}\mathbf{1}_{[t_j,t_{j+1})}(\tau_N^\rho)\Big]\\
& \leq \frac{C_0}{\rho}\E\Big[\|(V_N(t_j),\dt V_N(t_j))\|_{\H^{1-\eps}}\mathbf{1}_{[t_j,t_{j+1})}(\tau_N^\rho)\Big]+o(1),
\end{split}
\label{f5}
\end{align}

\noi
as $N \to \infty$.
When $j = 0$, we obtain \eqref{f4}
from \eqref{f5} since $(V_N(0),\dt V_N(0)) = (0, 0)$.
In general, 
by noting that 
 \[\|(V_N(t_j),\dt V_N(t_j)\|_{\H^{1-\eps}}
 \leq \|V_N\|_{X^{1-\eps}([t_{j-1}, t_j])},\] 
 
 \noi
 we apply the bound \eqref{f5} iteratively and obtain 
\begin{align*}
P(t_j\leq \tau_N^\rho <t_{j+1})
& \leq \frac{C_0}{\rho}\E\Big[
\|V_N\|_{X^{1-\eps}([t_{j-1}, t_j])}
\Big]+o(1)\\
& \leq \frac{C_0^2}{\rho}\E\Big[
\|(V_N(t_{j-1}),\dt V_N(t_{j-1})\|_{\H^{1-\eps}}
\Big]+o(1)\\
& \leq \cdots \leq 
 \frac{C_0^j}{\rho}\E\Big[
\|(V_N(0),\dt V_N(0)\|_{\H^{1-\eps}}
\Big]+o(1)\\
& \too 0
\end{align*}

\noi
as $N \to 0$
since $(V_N(0),\dt V_N(0)) = (0, 0)$.
This proves \eqref{f4}.
\end{proof}

\begin{remark}\label{REM:higher}\rm
As mentioned in Remark \ref{REM:x}
we can easily adapt the proof of Theorem \ref{THM:weak} presented above
to a general defocusing power-type nonlinearity $u^{2k+1}$, $k \in \N$, 
by using the following identity:
\[u_N^{2k+1} = \sum_{j=0}^{k}\binom{2k+1}{2j}(2j-1)!!\s_N^{j}\,:\!u_N^{2k+1-2j}\!:\]

\noi
in place of $u_N^3 = \, :\!u_N^3\!: + \,3\s_Nu_N$. 
Here, 
$(2j-1)!! = (2j-1)(2j-3)\cdots 3\cdot 1$
with the convention $(-1)!! = 1$.
In this case, the solution $u_N$ to
\[\begin{cases}(\dt^2-\Dl+\dt) u_N+ u_N^{2k+1} = \alpha_N\xi_N\\
(u_N,\dt u_N)\big|_{t=0} = (v_0,v_1)+(z_{0,N}^\o,z_{1,N}^\o)\end{cases}\]

\noi
 converges  to the  solution $w_{\kk}$ to
\[\begin{cases}(\dt^2-\Dl+\dt) w_{\kk} +\sum_{j=0}^{k}\binom{2k+1}{2j}(2j-1)!!
(\frac{\kk^2}{4\pi})^{j}w_{\kk}^{2k+1-2j} =0\\
(w_{\kk},\dt w_{\kk})\big|_{t=0} = (v_0,v_1)\end{cases}\]

\noi
as $N\rightarrow \infty$.
\end{remark}

\begin{ackno}\rm 
The authors would like to thank the anonymous referees for their helpful comments on the manuscript. 
T.O.~was supported by the European Research Council (grant no.~637995 ``ProbDynDispEq''
and grant no.~864138 ``SingStochDispDyn").  
T.R.~was supported by the European Research Council (grant no.~637995 ``ProbDynDispEq'').
M.O.~was supported by JSPS KAKENHI Grant number JP16K17624.
M.O.~would like to thank the School of Mathematics at the University of Edinburgh for its hospitality, where this manuscript was prepared.

\end{ackno}

\end{document}